%% file: mkv-sigma-39.tex
\documentclass[12pt,a4paper,reqno,oneside]{amsart}

\usepackage{amsmath}
\usepackage{amsthm}
\usepackage{amssymb}
\usepackage{bbm}
\usepackage[sort,nocompress]{cite} 
\usepackage{enumerate}
\usepackage{xcolor}
\usepackage{geometry}
\usepackage[colorlinks,citecolor=blue,linkcolor=blue]{hyperref}
\usepackage[notref,notcite,final]{showkeys}

\usepackage[linesnumbered,ruled,vlined]{algorithm2e}
\SetKwInput{KwInput}{Input}
\SetKwInput{KwOutput}{Output}

\usepackage{graphicx}
\usepackage[capitalise]{cleveref}
\usepackage{array}
\usepackage{etoolbox}
\AtBeginEnvironment{tabular}{\small}
\newcolumntype{L}[1]{>{\raggedright\let\newline\\\arraybackslash\hspace{0pt}}m{#1}}
\newcolumntype{R}[1]{>{\raggedleft\let\newline\\\arraybackslash\hspace{0pt}}m{#1}}

\newtheorem{lemma}{Lemma}
\newtheorem{remark}[lemma]{Remark}

\newtheorem{theorem}[lemma]{Theorem}

\theoremstyle{definition}

\renewcommand{\gets}{\curvearrowleft}
\providecommand{\N}{{\ensuremath{\mathbbm{N}}}}
\providecommand{\Z}{{\ensuremath{\mathbbm{Z}}}}
\providecommand{\R}{{\ensuremath{\mathbbm{R}}}}

\renewcommand{\P}{{\ensuremath{\mathbbm{P}}}}
\providecommand{\E}{{\ensuremath{\mathbbm{E}}}}

\providecommand{\1}{{\ensuremath{\mathbbm{1}}}}
\providecommand{\F}{{\ensuremath{\mathbbm{F}}}}

\newcommand{\threenorm}[1]{{\left\vert\kern-0.25ex\left\vert\kern-0.25ex\left\vert #1 
    \right\vert\kern-0.25ex\right\vert\kern-0.25ex\right\vert}}

\newcommand{\unif}{\mathfrak{u}}

\providecommand{\var}{{\ensuremath{\mathbbm{V}}}}
\newcommand{\rdown}[1] {\llcorner #1\lrcorner }

\newcommand{\vech}{\operatorname{vec}}

\author[A. Neufeld]{Ariel Neufeld}
\address{Division of Mathematical Sciences, Nanyang Technological University, Singapore}
\email{ariel.neufeld@ntu.edu.sg}

\thanks{Financial support by the MOE AcRF Tier 2 Grant \emph{MOE-T2EP20222-0013} is gratefully acknowledged.}

\author[T.A. Nguyen]{Tuan Anh Nguyen}
\address{Faculty of Mathematics, Bielefeld University, Germany}
\email{tnguyen@math.uni-bielefeld.de}

\author[P. Schmocker]{Philipp Schmocker}
\address{Department of Mathematics, ETH Zurich, Switzerland}
\email{philipp.schmocker@math.ethz.ch}

\title[MLP for McKean--Vlasov SDEs with nonconstant diffusion]{Multilevel Picard approximations\\for McKean--Vlasov stochastic differential equations\\with nonconstant diffusion
}
\allowdisplaybreaks
\begin{document}
	
\begin{abstract}
	We introduce multilevel Picard (MLP) approximations for McKean--Vlasov stochastic differential equations (SDEs) with nonconstant diffusion coefficient. Under standard Lipschitz assumptions on the coefficients, we show that the MLP algorithm approximates the solution of the SDE in the $L^2$-sense without the curse of dimensionality. The latter means that its computational cost grows at most polynomially in both the dimension and the reciprocal of the prescribed error tolerance. In two numerical experiments, we demonstrate its applicability by approximating McKean--Vlasov SDEs in dimensions up to 1000.
\end{abstract}
	
\maketitle

\section{Introduction}

McKean–Vlasov stochastic differential equations (SDEs) provide a mathematical framework to analyze the limiting behavior of stochastic systems of weakly interacting particles. In order to describe the limit, these SDEs  not only depend on the state of the process but also on their distribution, and are therefore also referred to as distribution-dependent SDEs or mean-field SDEs in the literature. First introduced by McKean in \cite{McKean1966} to model the Vlasov equation, these equations have since found applications in various fields, including statistical physics (see, e.g., \cite{Kac1956,Sznitman1991,DV1995,BCP1997}), mathematical biology (see, e.g., \cite{Graham1992,ME1999,BCM2007,BFFT2012,MSSZ2020}), social sciences (see, e.g., \cite{CFT2012,CJLW2017,GGSP2022}), mean-field games (see, e.g., \cite{CL2018,Carmona2018,CDLL2019}), stochastic control (see, e.g., \cite{BLM2017}), mathematical finance (see, e.g., \cite{GSS2020}), and the analysis of neural networks (see, e.g., \cite{MMN2018,SS2020,HRSS2021,RVE2022}).

More precisely, for $d \in \N$ and terminal time $T > 0$, we consider McKean--Vlasov SDEs defined for every $t \in [0,T]$ by
\begin{equation}
	\label{EqMcKV}
	X^d_\infty(t) = \xi_d+\int_{0}^{t}\E \!\left[\mu_d(x,X^d_\infty(s))\right]\!|_{x=X^d_\infty(s)}\,ds
	+ \int_{0}^{t}\E \!\left[\sigma_d(x,X^d_\infty(s))\right]\!|_{x=X^d_\infty(s)}\,dW^d(s),
\end{equation}
with initial value $\xi_d \in \R^d$, where $\mu_d: \R^d \times \R^d \rightarrow \R^d$ and $\sigma_d: \R^d \times \R^d \rightarrow \R^{d \times d}$ are sufficiently regular functions, and where $W^d = (W^d(t))_{t \in [0,T]}$ is a $d$-dimensional standard Brownian motion. The solution of \eqref{EqMcKV} can be seen as the limit of weakly interacting particles $(\mathcal{X}^{d,n}_N)_{n = 1,...,N}$, as $N \rightarrow \infty$, where each particle $\mathcal{X}^{d,n}_N := (\mathcal{X}^{d,n}_N(t))_{t \in [0,T]}$ satisfies the SDE defined for every $t \in [0,T]$ by
\begin{equation*}
	\mathcal{X}^{d,n}_N(t) = \xi_d + \frac{1}{N} \sum_{m=1}^N \int_0^t \mu_d(\mathcal{X}^{d,n}_N(s),\mathcal{X}^{d,m}_N(s)) ds
	+ \frac{1}{N} \sum_{m=1}^N \int_0^t \sigma_d(\mathcal{X}^{d,n}_N(s),\mathcal{X}^{d,m}_N(s)) dW^{d,n}(s).
\end{equation*}
This phenomenon is called the \emph{propagation of chaos} and was studied in \cite{Kac1956,McKean1967,Sznitman1991,Lacker2018,GGP2024}. While existence and uniqueness of \eqref{EqMcKV} have been established in \cite{Sznitman1991,MV2020,HSS2021,RZ2021}, the numerical approximation of \eqref{EqMcKV} has been considered in, e.g., \cite{BST2019,STT2019,RSZ2020optimalreg,RSZ2020pathreg,BH2021,BRRS2021,dRES2021,ST2021,BCNdRR2022,CdR2022,CNRS2022,LRS2022,RS2022,AAdRP2023,BRRS2023,CdR2023,CdRS2023,dRST2023,RSZ2023,Agarwal2021}.

In this paper, we are interested in the numerical approximation of \emph{high-dimensional} McKean–Vlasov SDEs. Our contribution is twofold: We introduce a multilevel Picard (MLP) approximation of \eqref{EqMcKV} and show that it overcomes the curse of dimensionality when approximating \eqref{EqMcKV} in the $L^2$-sense. The latter means that the computational cost of the MLP algorithm grows at most polynomially in the space dimension $d$ of the SDE \eqref{EqMcKV} and the reciprocal $\epsilon^{-1}$ of the prescribed $L^2$-error tolerance $\epsilon$. Second, we demonstrate the practical applicability of the MLP algorithm in two numerical examples by approximating McKean--Vlasov SDEs in dimensions up to 1000.

In the literature, only a few papers discuss the numerical approximation of high-dimensional McKean–Vlasov SDEs. First, \cite{HKN2022} has shown that the MLP algorithm does not suffer from the curse of dimensionality when approximating McKean--Vlasov SDEs with \emph{constant diffusion coefficient} $\sigma_d$, whence our paper can be seen as an extension to nonconstant diffusions. This is highly relevant for several applications across different fields: McKean--Vlasov SDEs with nonconstant diffusion arise in the stochastic Ginzburg-Landau equation in physics \cite{Nguyen2013}, in the Cucker–Smale model in flocking theory \cite{EHS2016}, as well as in spatially structured neuronal networks, the FitzHugh–Nagumo model, and the Hodgkin–Huxley model in neuronal biology \cite{BFFT2012,MSSZ2020}, to mention a few. Moreover, \cite{BHJM2021} introduced an MLP algorithm for approximating ordinary (instead of stochastic) differential equations with expectation, which overcomes the curse of dimensionality.

In recent years, \cite{GMW2022,HHL2024,PW2024,NN2025} have adopted deep learning techniques to approximately solve McKean–Vlasov SDEs of the form \eqref{EqMcKV}. While \cite{HHL2024} established an a priori error estimate under specific conditions (but not a full convergence and complexity analysis), only \cite{NN2025} proved that rectified deep neural networks do not suffer from the curse of dimensionality. However, the numerical examples presented in \cite{GMW2022,HHL2024,PW2024} are rather low-dimensional ($d \leq 15$), whereas the results in \cite{NN2025} remain purely theoretical without an explicit algorithm to approximate \eqref{EqMcKV}.

The following theorem is the main result of this paper. Hereby, we denote for every 
$d\in \N$ by $\lVert\cdot \rVert\colon \R^d \to [0,\infty)$ the standard norm on $\R^d$, i.e., 
for all $x=(x_i)_{i\in [1,d]\cap\Z}\in \R^d$ we have that
$\lVert x\rVert^2=\sum_{i=1}^{d}\lvert x_i\rvert^2$.
For every 
$d\in \N$ we also denote by $\lVert\cdot \rVert\colon \R^{d\times d} \to [0,\infty)$ the Hilbert-Schmidt norm on $\R^{d\times d}$, i.e., 
for all $x=(x_{ij})_{i,j\in [1,d]\cap\Z}\in \R^{d\times d}$ we have that
$\lVert x\rVert^2=\sum_{i,j=1}^{d}\lvert x_{ij}\rvert^2$.

\begin{theorem}\label{c01}
Let $T\in (0,\infty)$,  $c\in [1,\infty)$,  $\Theta=\cup_{n\in \N}(\N_0)^n$. 
For every $d\in \N$ let $\xi_d\in \R^d$,
$\mu_d \in C(\R^d\times\R^d,\R^d)$, 
$\sigma_d \in C(\R^d\times\R^d,\R^{d\times d})$. For every $K\in \N$ let $\rdown{\cdot}_K : [0,T] \rightarrow [0,T]$ satisfy for all $t\in (0,T]$ that 
$\rdown{t}_K=\sup( (0,t)\cap\{0,\frac{T}{K},\frac{2T}{K}\ldots,T\})$ and
$\rdown{0}_K=0$.
Assume for all $d\in \N$, $x_1,x_2,y_1,y_2 \in \R^d$ that
\begin{align}
\max \left\{
\left\lVert
\mu_d(x_1,x_2)-\mu_d(y_1,y_2)\right\rVert
,
\left\lVert
\sigma_d(x_1,x_2)-\sigma_d(y_1,y_2)\right\rVert
\right\}\leq 0.5c \lVert x_1-y_1\rVert+0.5c \lVert x_2-y_2\rVert,
\label{a11b}
\end{align}
\begin{align}
\lVert\xi_d\rVert+ \lVert\mu_d (0,0)\rVert+ \lVert \sigma_d(0,0) \rVert\leq cd^c.\label{a11c}
\end{align}
Let $(\Omega,\mathcal{F},\P, (\F_t)_{t\in [0,T]})$ be a filtered probability space which satisfies the usual conditions\footnote{Let $T \in [0,\infty)$ and ${\bf \Omega} = (\Omega,\mathcal{F},\P, (\F_t)_{t\in[0,T]})$ be 
	a filtered probability space. We say that ${\bf \Omega}$
	satisfies the usual conditions if and only if 
	$\{ A \subseteq \Omega: \exists N \in \mathcal{F} \text{ with } A \subseteq N \text{ and } \P(N) = 0 \} \subseteq \F_0$ and $\forall \, t\in [0,T) \colon \mathbb{F}_t 
	= \cap_{ s \in (t,T] } \F_s$.}. Let $\unif^\theta\colon \Omega\to [0,1]$, $\theta\in \Theta$, be independent and identically distributed random variables satisfying for all $\theta\in \Theta$, $t\in [0,1]$ that $\P(\unif^{\theta}\leq t)=t$. For every $d\in \N$ let $W^{d,\theta}=(W^{d,\theta}(t))_{t\in [0,T]}\colon [0,T]\times \Omega\to \R^d$, $\theta\in \Theta$, be independent standard  $(\F_t)_{t\in [0,T]}$-Brownian motions. Assume that $(\unif^\theta)_{\theta\in \Theta}$ and
$(W^{d,\theta})_{d\in \N,\theta\in \Theta}$ are independent. Let $X_{n,m,K}^{d,\theta}\colon [0,T]\times \Omega\to \R^d$, $\theta\in \Theta$, $d,m,K\in \N$, $n\in \N_0$, satisfy for all $\theta\in \Theta$,
$d,K,m,n\in \N$, $t\in [0,T]$ that 
$X^{d,\theta}_{0,m,K}(t)=0$ and
\begin{align} \begin{split} 
&
X^{d,\theta}_{n,m,K}(t)= \xi_d +t\mu_d(0,0)+\sigma_d(0,0)W^{d,\theta}(t)\\
&+ \sum_{\ell=1}^{n-1}\sum_{k=1}^{m^{n-\ell}}
\frac{t}{m^{n-\ell}}
\Bigl[
\mu_d (X_{\ell,m,K}^{d,\theta} (\rdown{t\unif^{(\theta,n,k,\ell)}}_K ),X_{\ell,m,K}^{d,(\theta,n,k,\ell)} (\rdown{t\unif^{(\theta,n,k,\ell)}}_K ))\\
&\qquad\qquad\qquad\qquad
-
\mu_d (X_{\ell-1,m,K}^{d,\theta} (\rdown{t\unif^{(\theta,n,k,\ell)}}_K ),X_{\ell-1,m,K}^{d,(\theta,n,k,\ell)} (\rdown{t\unif^{(\theta,n,k,\ell)}}_K ))
\Bigr]
\\
&+ \sum_{\ell=1}^{n-1}\sum_{k=1}^{m^{n-\ell}}\int_{0}^{t}
\frac{1}{m^{n-\ell}}
\Bigl[
\sigma_d(X_{\ell,m,K}^{d,\theta} (\rdown{s}_K ),X_{\ell,m,K}^{d,(\theta,n,k,\ell)} (\rdown{s}_K ))\\
&\qquad\qquad\qquad\qquad\qquad
-
\sigma_d(X_{\ell-1,m,K}^{d,\theta} (\rdown{s}_K ),X_{\ell-1,m,K}^{d,(\theta,n,k,\ell)} (\rdown{s}_K ))\Bigr]
dW^{d,\theta}(s).
\end{split}\label{a05c}
\end{align}
For every $d\in \N$ let  $\mathsf{cost}_{\mu_d} $ and 
$\mathsf{cost}_{\sigma_d} $ be an upper bound of the computational cost to calculate $\mu_d (x,y)$ and $\sigma_d(x,y)$
for any $x,y\in \R^d$. Let $\mathsf{cost}_\mathsf{rv}$ be 
an upper bound of 
the computational cost to generate a scalar random variable. For every $d,m,K\in \N$, $n\in \N_0$ let 
$\mathsf{C}_{n,m,K}^d$ denote the computational cost to construct the whole discrete process $(X_{n,m,K}^{d,\theta}(\frac{kT}{K}))_{k\in [1,K]\cap\Z}$ given that we have prepared a discrete Brownian path $W^{d,\theta}(\frac{kT}{K}))_{k\in [1,K]\cap\Z} $.
Assume for all $d\in \N$ that
\begin{align}
\mathsf{cost}_{\mu_d} +
\mathsf{cost}_{\sigma_d}+\mathsf{cost}_\mathsf{rv}\leq cd^c.\label{c10}
\end{align}
Then the following holds:
\begin{enumerate}[(i)]
\item\label{b02} For all $d\in \N$, $\theta\in \Theta$ there exists
a unique $\R^d$-valued $(\F_t)_{t\in [0,T]}$-adapted process $X_\infty^{d,\theta}=(X_\infty^{d,\theta}(t))_{t\in [0,T]}$ with continuous sample paths such that we have $\P$-a.s.\ for all $t\in [0,T]$ that $
\sup_{s\in [0,T]}\lVert X_\infty^{d,\theta}(s)\rVert_{L^2(\P;\R^d)}<\infty
$ and
\begin{align}
	\label{b02c}
	X_\infty^{d,\theta}(t) =\xi_d+\int_{0}^{t}\E \!\left[\mu_d(x,X_\infty^{d,\theta}(s))\right]\!|_{x=X_\infty^{d,\theta}(s)}\,ds
	+
	\int_{0}^{t}\E \!\left[\sigma_d(x,X_\infty^{d,\theta}(s))\right]\!|_{x=X_\infty^{d,\theta}(s)}\,dW^{d,\theta}(s).
\end{align}
\item \label{c19}
There exist $(\mathfrak{C}_\delta)_{\delta\in (0,1)}\subseteq \R$,
$(\mathsf{n}_{d,\epsilon})_{d\in \N,\epsilon\in (0,1)}\subseteq \N$ such that for all $d\in \N$, $\delta,\epsilon\in (0,1)$ we have that
\begin{align}
	\label{c19c}
\sup_{k\in [\mathsf{n}_{d,\epsilon},\infty)\cap\Z}
\sup _{t\in [0,T]}
\left\lVert
X^{d,0}_{k,k,k^k}(t)-X^{d,0}_\infty(t)
\right\rVert_{L^2(\P;\R^d)}<\epsilon
\end{align}
and 
\begin{align}
	\label{c19d}
\mathsf{C}_{\mathsf{n}_{d,\epsilon},\mathsf{n}_{d,\epsilon}, (\mathsf{n}_{d,\epsilon})^{\mathsf{n}_{d,\epsilon}}}^d<(d^{c+1})^6\mathfrak{C}_\delta\epsilon^{-(4+\delta)}.
\end{align}
\end{enumerate}
\end{theorem}

\begin{remark}\label{x01}
	Note that the McKean--Vlasov SDEs \eqref{b02c} can be equivalently written using the laws of $(X_\infty^{d,\theta}(s))_{s \in [0,T]}$, instead of the expectation. More precisely, for every fixed $d \in \N$, we denote by $P_2(\R^d)$ the set of probability measures $\nu$ on $\R^d$ satisfying $\int_{\R^d} x^2 \nu(dx) < \infty$ and define the functions $\R^d \times P_2(\R^d) \ni (x,\nu) \mapsto \widetilde{\mu}_d(x,\nu) := \int_{\R^d} \mu_d(x,y) \nu(dy) \in \R^d$ and $\R^d \times P_2(\R^d) \ni (x,\nu) \mapsto \widetilde{\sigma}_d(x,\nu) := \int_{\R^d} \sigma_d(x,y) \nu(dy) \in \R^{d \times d}$. Then, it holds $\P$-a.s.\ for all $\theta \in \Theta$, $t\in [0,T]$ that
	\begin{equation*}
		\begin{aligned}
			X_\infty^{d,\theta}(t)&=\xi_d+\int_{0}^{t}\E \!\left[\mu_d(x,X_\infty^{d,\theta}(s))\right]\!|_{x=X_\infty^{d,\theta}(s)}\,ds
			+
			\int_{0}^{t}\E \!\left[\sigma_d(x,X_\infty^{d,\theta}(s))\right]\!|_{x=X_\infty^{d,\theta}(s)}\,dW^{d,\theta}(s) \\
			&=\xi_d+\int_{0}^{t} \widetilde{\mu}_d(X_\infty^{d,\theta}(s),\nu_{X_\infty^{d,\theta}(s)}) ds
			+
			\int_{0}^{t} \widetilde{\sigma}_d(X_\infty^{d,\theta}(s),\nu_{X_\infty^{d,\theta}(s)})\,dW^{d,\theta}(s),
		\end{aligned}
	\end{equation*}
	where $\nu_{X_\infty^{d,\theta}(s)} \in P_2(\R^d)$ denotes the law of the random variable $X_\infty^{d,\theta}(s)$, which is defined as $\nu_{X_\infty^{d,\theta}(s)}[A] := \P[X_\infty^{d,\theta}(s) \in A]$ for all Borel-subsets $A \subseteq \R^d$.
\end{remark}
\begin{remark}\label{r05}
Under the assumptions of Theorem~\ref{c01}, 
for all $\theta\in \Theta$,
$d,K,m,n\in \N$, 
$\ell\in[1,n-1]\cap\Z $, $k\in [1,m^{n-\ell}]\cap\Z$,
$t\in [0,T]$, $J\in [0,K-1]\cap\Z$ with
$t=\frac{JT}{K}$ we have
 that 
\begin{align}
&
\int_{0}^{t}\left[\sigma_d(X_{\ell,m,K}^{d,\theta} (\rdown{s}_K ),X_{\ell,m,K}^{d,(\theta,n,k,\ell)} (\rdown{s}_K ))
-\sigma_d(X_{\ell-1,m,K}^{d,\theta} (\rdown{s}_K ),X_{\ell-1,m,K}^{d,(\theta,n,k,\ell)} (\rdown{s}_K )
\right]
dW^{d,\theta}(s)\nonumber\\
&=
\sum_{j=0}^{J-1}\left[
\sigma_d(X_{\ell,m,K}^{d,\theta} (\tfrac{jT}{K} ),X_{\ell,m,K}^{d,(\theta,n,k,\ell)} 
(\tfrac{jT}{K}))
-
\sigma_d(X_{\ell-1,m,K}^{d,\theta} (\tfrac{jT}{K} ),X_{\ell-1,m,K}^{d,(\theta,n,k,\ell)} (\tfrac{jT}{K}))\right]
(W^{d,\theta}(\tfrac{(j+1)T}{K})-W^{d,\theta}(\tfrac{jT}{K}) ).
\end{align}
Thus, we see that the stochastic integral in \eqref{a05c} is time-discretized and hence implementable.
\end{remark}%
Theorem~\ref{c01} proves mathematically that the multilevel Picard (MLP) approximation \eqref{a05c} overcomes the curse of dimensionality when approximating the McKean--Vlasov SDE \eqref{b02c}. Indeed, by \eqref{c19c} and \eqref{c19d}, the computational cost $\mathsf{C}_{n,m,K}^d$ of the MLP algorithm grows at most polynomially in both the state-space dimension $d$ and the reciprocal $\epsilon^{-1}$ of the prescribed $L^2$-error tolerance $\epsilon$.

Let us give some motivations for the numerical schema \eqref{a05c}. First, we start with an iteration defined as follows $X^d_{0}(t)=0$ and
\begin{align}
X^d_n(t) &= \xi_d+\int_{0}^{t}\E \!\left[  \mu_d(x,X^d_{n-1}(s))\right]\!|_{x=X^d_{n-1}(s)}\,ds
+ \int_{0}^{t}\E \!\left[\sigma_d(x,X^d_{n-1}(s))\right]\!|_{x=X^d_{n-1}(s)}\,dW^d(s)
\end{align} for $n\in \N$. Note that the r.h.s. can be written as a telescoping sum: 
\begin{align}
 X^d_n(t)   &=
\xi_d+t\mu_d(0,0)+\sigma_d(0,0)W^{d,\theta}(t)\nonumber\\&\quad +\int_{0}^{t}\sum_{\ell=1}^{n-1}\E \!\left[  \mu_d(x,X^d_{\ell}(s))\right]\!|_{x=X^d_{\ell}(s)}-
\E \!\left[  \mu_d(x,X^d_{\ell-1}(s))\right]\!|_{x=X^d_{\ell-1}(s)}
\,ds\nonumber\\
&\quad 
+ \int_{0}^{t}\sum_{\ell=1}^{n-1}\E \!\left[\sigma_d(x,X^d_{\ell}(s))\right]\!|_{x=X^d_{\ell}(s)}
-\E \!\left[\sigma_d(x,X^d_{\ell-1}(s))\right]\!|_{x=X^d_{\ell-1}(s)}
\,dW^d(s)
\end{align} where the index $\ell$ can be interpreted as the corresponding level.
 By replacing here all expectations and integrals by corresponding Monte-Carlo approximations
 we derive our multilevel Picard approximation scheme \eqref{a05c}. More precisely, the ideas can be describe as follows. First, we approximate an integral $\int_0^t h(s)\,ds$ for a random function $h$ by 
 $th(\mathfrak{u})$ where $\mathfrak{u}$ is uniformly distributed on $[0,1]$ and independent of $h$.  
 In order to approximate the expectations
 $
 \E \!\left[  \mu_d(x,X^d_{\ell-1}(s))\right]\!|_{x=X^d_{\ell-1}(s)}$ and
 $
 \E \!\left[  \sigma_d(x,X^d_{\ell-1}(s))\right]\!|_{x=X^d_{\ell-1}(s)}$
 we take 
 one sample $X^{d,\theta}_{\ell,m,K}$, which is also generated by the Brownian motion $W^{d,\theta}$, and
 $m^{n-\ell}$ independent samples $X^{d,(\theta,n,k,\ell)}$, $k\in \{1,2,\ldots,m^{n-\ell}\}$, which are respectively generated by the independent Brownian motions $W^{d,(\theta,n,k,\ell)}$, $k\in \{1,2,\ldots,m^{n-\ell}\}$. 
The main novelty of the numerical scheme in \eqref{a05c} in  Theorem~\ref{c01} compared to \cite[Theorem~1.1]{HKN2022} lies in storing the entire process $(X_{n,m,K}^{d,\theta}(t))_{t\in [0,T]}$ rather than only one 
random variable $X_{n,m,K}^{d,\theta}(t)$. Indeed, by retaining the whole path $(X_{n,m,K}^{d,\theta}(t))_{t\in [0,T]}$, we can approximate the stochastic integral arising in McKean--Vlasov SDEs with nonconstant diffusion coefficient, which was not necessary for the constant diffusion case in \cite{HKN2022} where the stochastic integral is just a normally distributed random variable.
Since we cannot save a process $Y$ at all time points 
$t\in [0,T]$ we only save the discrete trajectory $(Y(\rdown{t}_K))_{t\in [0,T]}$.
Here and in \eqref{a05c} the index $K$ denote the number of steps of the time discretization to approximate the stochastic integrals. 
While this approach may seem inefficient at first sight, our numerical experiments demonstrate its effectiveness even in high dimensions. Moreover, for many applications of McKean--Vlasov SDEs, one requires the computation of the entire solution trajectory $(X^{d,\theta}_{\infty}(t))_{t\in [0,T]}$ instead of one single space-time point $X^{d,\theta}_{\infty}(t)$, which our method can do.
\begin{remark}\label{r01} In Theorem~\ref{c01}, the higher complexity of order $\epsilon^{-(4+\delta)}$ compared to $\epsilon^{-(2+\delta)}$ in the case when $\sigma$ is constant is due to the fact that we have to approximate the stochastic integrals and therefore need to save the whole trajectory of the approximation solution. As seen in \eqref{c04}, the $L^2$ error consists of two parts: an MLP error (the first term)
and a discretization error (the second term). These two errors should be balanced. This leads to the choice $m=n$ and $K=n^n$. As a result, we obtain a computational complexity of order $\epsilon^{-(4+\delta)}$.  We believe that this computational complexity is not sharp. As a future research we would like to design an algorithm which can overcome the curse of dimensionality with computational complexity of order $\epsilon^{-(2+\delta)}$, the same computational complexity as in the case $\sigma$ is constant. 
\end{remark}%
\begin{remark}\label{r04} In Theorem~\ref{c01} the $dt$-integral can also be approximated by time discretization instead of uniform Monte-Carlo sampling. It can be  seen by slightly adapting the proof that the computational complexity is still of order $\epsilon^{-(4+\delta)}$. 
However, in practice,  the algorithm runs faster when we use uniform Monte-Carlo sampling because it does not need to evaluate $\mu_d$ along the whole discrete trajectory but only one random point on it. In the time integral we need to introduce the rounding because we cannot save the whole trajectory of the approximation solution
$(X_{n,m,K}(t))_{t\in [0,T]}$ but only its discrete part 
$(X_{n,m,K}(kT/K))_{k\in [0,K]\cap\mathbb{Z}}$.
\end{remark}%
\begin{remark}
	\label{RemLowLevels}
	Note that the MLP scheme~\eqref{a05c} outputs only a single trajectory approximating the solution to the McKean-Vlasov SDE~\eqref{b02c}. However, the lower levels $X_{\ell,m,K}^{d,(\theta,n,k,\ell)}$, $k \in \lbrace 1,...,m^{n-\ell} \rbrace$ and $\ell \in \lbrace 1,...,n-1 \rbrace$, of the MLP algorithm can be used to approximate additional statistical quantities. For example, we use the lower levels in the numerical example of Section~\ref{SecKuramoto} to approximate the expectation in the drift coefficient of the McKean-Vlasov SDE~\eqref{b02c}.
\end{remark}
The remainder of this article is organized as follows. In Section~\ref{SecNumerics}, we provide two numerical examples to approximate the solution of \eqref{EqMcKV} in dimension up to 1000, followed by the conclusion in Section~\ref{SecConclusion}. In Section~\ref{SecPicard}, we show an auxiliary result on Picard iterations of the McKean--Vlasov SDE \eqref{EqMcKV}, while Section~\ref{SecMLP} contains the proof of the main result in Theorem~\ref{c01}.

\section{Numerical experiments}
\label{SecNumerics}

In this section, we provide two numerical examples\footnote{The numerical experiments have been implemented in \texttt{Python} and executed on a HPC cluster of D-MATH, ETH Zurich. The code is available under the following link: \url{https://github.com/psc25/McKeanVlasovMLP}.} to illustrate how the solution of the McKean-Vlasov SDE \eqref{EqMcKV} can be approximated by the multilevel Picard (MLP) scheme \eqref{a05c}. First, let us summarize the MLP scheme~\eqref{a05c} as pseudocode in Algorithm~\ref{AlgMLP}.

\begin{algorithm}[!ht]
	\DontPrintSemicolon
	\begin{small}
		\KwInput{$K,m \in \mathbb{N}$, $n \in \mathbb{N}_0$, $T > 0$, $\xi_d \in \R^d$, $\mu^d \in C(\R^d \times \R^d,\R^{1 \times d})$, and $\sigma^d \in C(\R^d \times \R^d,\R^{d \times d})$.}
		\KwOutput{MLP approximation $X = (X_{k,i})_{k=0,...,K}^{i=1,...,d}$ of the solution to \eqref{b02c} at times $(\frac{kT}{K})_{k=0,...,K}$.}
		
		\SetKwProg{Fn}{Function}{}{end}
		
		$\Delta t \leftarrow T/K$
		
		$\mathbf{t}_K \leftarrow (k \Delta j)_{j=0,...,K} \in \R^{(K+1) \times 1}$
		
		Generate $\R^{K \times d}$-valued realization $\Delta W \sim \mathcal{N}_{K \times d}(0,\sqrt{\Delta t})$.
		
		\Fn{\texttt{MLP}($\xi_d := (\xi_d^i)_{i=1,...,d} \in \R^{1 \times d},m \in \mathbb{N},n \in \mathbb{N}_0,\Delta W := (\Delta W_{j,i})_{j=1,...,K}^{i=1,...,d} \in \R^{K \times d}$)}{
			
			\If{$n = 0$}{
				\KwRet $0 \in \R^{(K+1) \times d}$ 
			}
			
			\Else{
				$X := (X_j)_{j=0,...,K} := (X_{j,i})_{j=0,...,K}^{i=1,...,d} \leftarrow (\xi_d^i)_{j=0,...,K}^{i=1,...,d} \in \R^{(K+1) \times d}$
				
				$X \leftarrow X + \mathbf{t}_K \mu^d(0, 0) \in \R^{(K+1) \times d}$
				
				$X \leftarrow X + \big( \sum_{k=1}^j \Delta W_{k,i} \big)_{j=0,...,K}^{i=1,...,d} \sigma^d(0,0)^\top \in \R^{(K+1) \times d}$
				
				\For{$l = 1,...,n-1$}{
					\For{$k = 1,...,m^{n-l}$}{
						Generate $\R^{K \times d}$-valued realization $\Delta \widetilde{W} := (\Delta \widetilde{W}_{j,i})_{j=1,...,K}^{i=1,...,d} \sim \mathcal{N}_{K \times d}(0,\sqrt{\Delta t})$.
						
						$X^{(1)} := (X^{(1)}_j)_{j=0,...,K} := (X^{(1)}_{j,i})_{j=0,...,K}^{i=1,...,d} \leftarrow \texttt{MLP}(\xi, l, \Delta W)$
						
						$X^{(2)} := (X^{(2)}_j)_{j=0,...,K} := (X^{(2)}_{j,i})_{j=0,...,K}^{i=1,...,d} \leftarrow \texttt{MLP}(\xi, l, \Delta \widetilde{W})$
						
						$X^{(3)} := (X^{(3)}_j)_{j=0,...,K} := (X^{(3)}_{j,i})_{j=0,...,K}^{i=1,...,d} \leftarrow \texttt{MLP}(\xi, l-1, \Delta W)$
						
						$X^{(4)} := (X^{(4)}_j)_{j=0,...,K} := (X^{(4)}_{j,i})_{j=0,...,K}^{i=1,...,d} \leftarrow \texttt{MLP}(\xi, l-1, \Delta \widetilde{W})$
						
						$I := (I_j)_{j=0,...,K} \leftarrow 0 \in \R^{(K+1) \times d}$
						
						\For{$j = 1,...,K$}{
							$I_j \leftarrow I_{j-1} + \frac{1}{m^{n-l}} \big( \sigma^d(X^{(1)}_{j-1},X^{(2)}_{j-1}) - \sigma^d(X^{(3)}_{j-1},X^{(4)}_{j-1}) \big) \Delta W_j \in \R^d$
						}
						
						$X \leftarrow X + I \in \R^{(K+1) \times d}$
						
						Generate a real-valued realization $u \sim \mathcal{U}(0,1)$.
						
						\For{$j = 0,...,K$}{
							$r \leftarrow \lfloor (j-1) u \rfloor+1 \in \R$
							
							$X_j \leftarrow X_j + \frac{1}{m^{n-l}} \big( \mu^d(X^{(1)}_r,X^{(2)}_r) - \mu^d(X^{(3)}_r,X^{(4)}_r) \big) \in \R^d$
						}
					}
				}
				\KwRet $X \in \R^{(K+1) \times d}$ 
			}
		}
		
		$X \leftarrow \texttt{MLP}(\xi_d,m,n,\Delta W) \in \R^{(K+1) \times d}$
		
		\KwRet $X \in \R^{(K+1) \times d}$ 
	\end{small}
	\caption{MLP algorithm for McKean-Vlasov stochastic differential equations}
	\label{AlgMLP}
\end{algorithm}

\subsection{Mean-field Ornstein-Uhlenbeck model}

In the first example, we consider an Ornstein-Uhlenbeck process with a mean-field term in the drift and diffusion coefficient (see \cite[Section~5.3]{AO2023}). More precisely, we assume that the $(\mathbb{F}_t)_{t \in [0,T]}$-adapted process $X^{d,\theta}_\infty = (X^{d,\theta}_\infty(t))_{t \in [0,T]}$ for every $t \in [0,T]$ is given by
\begin{equation}
	\label{EqDefMfOU}
	\begin{aligned}
		X^{d,\theta}_\infty(t) & = \xi_d + \int_0^t \left( a_0 + A_1 X^{d,\theta}_\infty(s) + A_2 \E\!\left[ X^{d,\theta}_\infty(s) \right] \right) ds + \sum_{k=1}^d \int_0^t \left( b_k + B_k \E\!\left[ X^{d,\theta}_\infty(s) \right] \right) dW^{d,\theta,k}(s),
	\end{aligned}
\end{equation}
where $\xi_d \in \R^d$ is the initial value, where $a_0,b_1,...,b_d \in \R^d$ are vectors, and where $A_1, A_2, B_1,...,B_d \in \R^{d \times d}$ are matrices. Comparing to Theorem~\ref{c01}, this corresponds to the drift $\mu_d(x_1,x_2) = a_0 + A_1 x_1 + A_2 x_2$ and diffusion coefficient $\sigma_d(x_1,x_2) = (b_k+B_k x_2)_{k=1,...,d}$ satisfying \eqref{a11b} and \eqref{a11c}. Moreover, by taking the expectation in \eqref{EqDefMfOU} and using Fubini's theorem, we have
\begin{equation*}
	\E\!\left[ X^{d,\theta}_\infty(t) \right] = \xi_d + a_0 t + (A_1 + A_2) \int_0^t \E\!\left[ X^{d,\theta}_\infty(s) \right] ds,
\end{equation*}
whose solution is given by
\begin{equation}
	\label{EqMeanMfOU}
	m_d(t) := \E\!\left[ X^{d,\theta}_\infty(t) \right] = e^{(A_1+A_2)t} \xi_d + \int_0^t e^{(A_1+A_2)(t-s)} a_0 ds.
\end{equation}
For different dimensions $d \in \lbrace 10, 50, 100, 500, 1000 \rbrace$ and levels $n = m \in \lbrace 1,...,5 \rbrace$ with $K = m^n$, we run the algorithm $10$ times with terminal time $T = 1$, initial value $\xi_d = (20,...,20)^\top$, and randomly initialized parameters that are fixed over the different levels and runs, i.e.,
\begin{align*}
	a & \sim \mathcal{N}_d\big( 0,(20\sqrt{d})^{-1} I_d \big), & & & \vech(A_k) & \sim \mathcal{N}_{d^2}\big( 0,(100d)^{-1} I_{d^2} \big) & & k \in \lbrace 1,2 \rbrace \\
	b_k & \sim \mathcal{N}_d\big( 0,(5d)^{-1} I_d \big), & & k \in \lbrace 1,...,d \rbrace, \quad & \vech(B_k) & \sim \mathcal{N}_{d^2}\big( 0,(5d)^{-1} I_{d^2} \big) & & k \in \lbrace 1,...,d \rbrace,
\end{align*}
which are all independent of each other, where $\vech(A) := (a_{1,1},a_{1,2},...,a_{d,d})^\top \in \R^{d^2}$ for $A := (a_{i,j})_{i=1,...,d}^{j=1,...,d}$, where $\mathcal{N}_d(0,\Sigma)$ denotes the normal distribution with zero mean and covariance matrix $\Sigma \in \R^{d \times d}$, and where $I_d \in \R^{d \times d}$ is the identity matrix. The results are reported in Table~\ref{TabMfOU} with ($d^{-1/2}$-adjusted) $L^2$-error given by
\begin{equation}
	\label{EqDefL2Err}
	\left( \frac{1}{10 d K} \sum_{r=1}^{10} \sum_{k=1}^K \left\Vert X^{d,\theta}_\infty(\omega_r)(t_k) - X^{d,\theta}_{n,m,K}(\omega_r)(t_k) \right\Vert^2 \right)^\frac{1}{2}.
\end{equation}
Here, the true solution $X^{d,\theta}_\infty$ is approximated by an Euler-Maruyama scheme of \eqref{EqDefMfOU} over a finer equidistant time grid $(j T/K_1)_{j=0,...,K_1}$ with $K_1 := K \lfloor\frac{500}{K} \rfloor + 1$ points. The realizations of the Brownian increments $\big(\Delta W^{d,\theta,k}(j T/K_1)\big)_{j=0,\dots,K_1}$ on this finer Euler-Maruyama grid are then summed over each time subinterval of the coarser MLP grid to obtain the Brownian increments $\big(\Delta W^{d,\theta,k}(j T/K)\big)_{j=0,\dots,K}$ used in the MLP approximation.

\begin{table}[h!]
	\begin{small}
		\begin{tabular}{rl|R{2.3cm}R{2.3cm}R{2.3cm}R{2.3cm}R{2.3cm}|}
			& & \multicolumn{5}{c|}{Level $n$, samples per level $m$, and time points $K = m^n$} \\
			$d$ & & $n = m = 1$ & $n = m = 2$ & $n = m = 3$ & $n = m = 4$ & $n = m = 5$ \\
			\hline
			\input{ou_table.txt}
		\end{tabular}
		\caption{MLP solution of the mean-field Ornstein-Uhlenbeck model \eqref{EqDefMfOU} for different $d \in \lbrace 10, 50, 100, 500, 1000 \rbrace$, $n = m \in \lbrace 1,...,5 \rbrace$, and $10$ independent runs of the algorithm. While the $L^2$-error \eqref{EqDefL2Err} is displayed in the rows ``$L^2$-Error'', the rows ``Time'' and ``Cost'' report the computational times (in seconds) and the computational costs $\mathsf{C}_{n,m,K}^d$, respectively, which are both averaged over the $10$ runs.}
		\label{TabMfOU}
	\end{small}
\end{table}

\subsection{Multidimensional geometric Kuramoto model}
\label{SecKuramoto}

In the second example, we consider a multidimensional geometric version of the Kuramoto model, whose original model has been studied in \cite{S2000,ABPRS05,CGPS2020,dRST2023} and the references therein. More precisely, we assume that the $(\mathbb{F}_t)_{t \in [0,T]}$-adapted process $X^{d,\theta}_\infty = (X^{d,\theta}_\infty(t))_{t \in [0,T]}$ for every $t \in [0,T]$ is given by
\begin{equation}
	\label{EqDefKuramoto}
	X^{d,\theta}_\infty(t) = \xi_d + \mu_0 \int_0^t \E\!\left[ \sin_d\left( x - X^{d,\theta}_\infty(s) \right) \right]\!|_{x = X^{d,\theta}_\infty(s)} ds + \sum_{k=1}^d \int_0^t \Sigma_k X^{d,\theta}_\infty(s) dW^{d,\theta,k}(s),
\end{equation}
where $\xi_d \in \R^d$ is the initial value, where $\mu_0 \in \R$ and $\Sigma_1,...,\Sigma_d \in \R^{d \times d}$ are some parameters, and where $\R^d \ni x = (x_1,...,x_d)^\top \mapsto \sin_d(x) := (\sin(x_1),...,\sin(x_d))^\top \in \R^d$ denotes the componentwise sine function. Comparing to Theorem~\ref{c01}, this corresponds to the drift $\mu_d(x_1,x_2) = \mu_0 \sin_d(x_1-x_2)$ and diffusion coefficient $\sigma_d(x_1,x_2) = (\Sigma_k x_1)_{k=1,...,d}$ satisfying \eqref{a11b} and \eqref{a11c}. Moreover, we have
\begin{equation}
	\label{EqExpKuramoto}
	\E\!\left[ \sin_d\!\left( x \!-\! X^{d,\theta}_\infty(s) \right) \right] \!|_{x = X^{d,\theta}_\infty(s)} = \sin_d\!\left( X^{d,\theta}_\infty(s) \right) \E\!\left[ \cos_d\!\left( X^{d,\theta}_\infty(s) \right) \right] - \cos_d\!\left( X^{d,\theta}_\infty(s) \right) \E\!\left[ \sin_d\!\left( X^{d,\theta}_\infty(s) \right) \right].
\end{equation}
For different dimensions $d \in \lbrace 10, 50, 100, 500, 1000 \rbrace$ and levels $n = m \in \lbrace 1,...,5 \rbrace$ with $K = m^n$, we run the algorithm $10$ times with terminal time $T = 1$, initial value $\xi_d = (\xi_d^1,...,\xi_d^d)^\top = (10,...,10)^\top$, parameter $\mu_0 = 1$, and randomly initialized parameters $(\Sigma_k)_{k=1,...,d}$ that are fixed over the different levels and runs, i.e., samples from the independent random variables
	\begin{align*}
		\vech(\Sigma_k) \sim \mathcal{N}_{d^2}\big( 0, (10d)^{-1} I_{d^2} \big), \qquad k \in \lbrace 1,...,d \rbrace.
	\end{align*}
The results are reported in Table~\ref{TabKuramoto}, where the true solution $X^{d,\theta}_\infty = (X^{d,\theta}_\infty(t))_{t \in [0,T]}^\top$ is approximated by an Euler-Maruyama scheme of \eqref{EqDefKuramoto} over a finer equidistant time grid $(j T/K_1)_{j=0,...,K_1}$ with $K_1 := K \lfloor \frac{500}{K} \rfloor + 1$ points. For the drift term of \eqref{EqDefKuramoto}, the expectation~\eqref{EqExpKuramoto} is approximated by an average over the lower levels of the MLP algorithm, which are then extended from $(j T/K)_{j=0,...,K}$ to $(j T/K)_{j=0,...,K_1}$ by piecewise constant interpolation. Again, the realizations of the Brownian increments $\big(\Delta W^{d,\theta,k}(j T/K_1)\big)_{j=0,\dots,K_1}$ on the finer Euler-Maruyama grid are then summed over each time subinterval of the coarser MLP grid to obtain the Brownian increments $\big(\Delta W^{d,\theta,k}(j T/K)\big)_{j=0,\dots,K}$ used in the MLP approximation.

\begin{table}[h!]
	\begin{small}
		\begin{tabular}{rl|R{2.3cm}R{2.3cm}R{2.3cm}R{2.3cm}R{2.3cm}|}
			& & \multicolumn{5}{c|}{Level $n$, samples per level $m$, and time points $K = m^n$} \\
			$d$ & & $n = m = 1$ & $n = m = 2$ & $n = m = 3$ & $n = m = 4$ & $n = m = 5$ \\
			\hline
			\input{ku_table.txt}
		\end{tabular}
		\caption{MLP solution of the multidimensional geometric Kuramoto model \eqref{EqDefKuramoto} for different $d \in \lbrace 10, 50, 100, 500, 1000 \rbrace$, $n = m \in \lbrace 1,...,5 \rbrace$, and $10$ independent runs of the algorithm. While the $L^2$-error \eqref{EqDefL2Err} is displayed in the rows ``$L^2$-Error'', the rows ``Time'' and ``Cost'' report the computational times (in seconds) and the computational costs $\mathsf{C}_{n,m,K}^d$, respectively, which are both averaged over the $10$ runs.}
		\label{TabKuramoto}
	\end{small}
\end{table}

\vspace{0.5cm}
\section{Conclusion}
\label{SecConclusion}

Theorem~\ref{c01} proves mathematically that the multilevel Picard (MLP) approximation in \eqref{a05c} overcomes the curse of dimensionality when approximating the McKean--Vlasov stochastic differential equation (SDE) in \eqref{EqMcKV}. This means that the computational cost of the MLP algorithm grows at most polynomially in both the state-space dimension $d$ and the reciprocal $\epsilon^{-1}$ of the prescribed $L^2$-error tolerance $\epsilon$. Moreover, the two numerical examples demonstrate the practical applicability of the MLP algorithm for the approximation of high-dimensional McKean--Vlasov SDEs. In particular, Table~\ref{TabMfOU} and Table~\ref{TabKuramoto} confirm empirically that the MLP approximation overcomes the curse of dimensionality as the ($d^{-1/2}$-adjusted) $L^2$-error defined in \eqref{EqDefL2Err} does not increase significantly in the dimension $d$, while the computational cost seem to increase linearly in $d$.

\newpage
\section{An auxiliary lemma}
\label{SecPicard}
In this section, we show an auxiliary result for
solutions to McKean--Vlasov SDEs of the form~\eqref{EqMcKV} and its discretization which will be useful for the proof of Theorem~\ref{c01}.

\begin{lemma}\label{c12}
Let $T\in (0,\infty)$, 
$c\in [1,\infty)$,
$d\in\N $, 
$\xi_d\in \R^d$,
 $\mu 
\in C(\R^d\times \R^d,\R^d)$, $\sigma\in C(\R^d\times \R^d,\R^{d\times d})$ satisfy for all $x_1,x_2,y_1,y_2\in \R^d$ that
\begin{align}
\max \left\{
\left\lVert
\mu(x_1,x_2)-\mu(y_1,y_2)\right\rVert
,
\left\lVert
\sigma(x_1,x_2)-\sigma(y_1,y_2)\right\rVert
\right\}\leq 0.5c \lVert x_1-y_1\rVert+0.5c \lVert x_2-y_2\rVert.
\label{a11}
\end{align}
For every $K\in \N$ let $\rdown{\cdot}_K : [0,T] \rightarrow [0,T]$ satisfy for all $t\in (0,T]$ that 
$\rdown{t}_K=\sup( (0,t)\cap\{0,\frac{T}{K},\frac{2T}{K}\ldots,T\})$ and
$\rdown{0}_K=0$. For convenience we write for every $t\in [0,T]$ that  $\rdown{t}_\infty=t $.
Let $(\Omega,\mathcal{F}, \P, (\F_t)_{t\in [0,T]})$ be a filtered probability space which satisfies the usual conditions. Let 
$W=(W(t))_{t\in [0,T]}\colon [0,T]\times\Omega\to \R^d$ be a standard $(\F_t)_{t\in [0,T]}$-Brownian motion.
Then the following holds:
\begin{enumerate}[(i)]
\item\label{a04} 
For all $K\in \N\cup\{\infty\}$ there exists
a unique $(\F_t)_{t\in [0,T]}$-adapted process $X_K=(X_K(t))_{t\in [0,T]}$ with continuous sample paths such that we have $\P$-a.s.\ for all $t\in [0,T]$ that $
\sup_{s\in [0,T]}\lVert X_K(s)\rVert_{L^2(\P,\R^d)}<\infty
$ and
\begin{align}
X_K(t)&=\xi_d+\int_{0}^{t}\E \!\left[\mu(x,X_K(\rdown{s}_K))\right]\!|_{x=X_K(\rdown{s}_K)}\,ds \, +
\int_{0}^{t}\E \!\left[\sigma(x,X_K(\rdown{s}_K))\right]\!|_{x=X_K(\rdown{s}_K)}\,dW(s).
\end{align}
\item\label{a09} For all $t\in [0,T]$, $K\in \N\cup\{\infty\}$ we have that
\begin{align}
\left\lVert X_K(t)\right\rVert_{L^2(\P,\R^d)}
\leq \sqrt{2}\left(
\lVert\xi_d\rVert+(\sqrt{T}+1)
\max\{\lVert\mu(0,0)\rVert,\lVert\sigma(0,0)\rVert\}
\right)e^{(\sqrt{T}+1)^2c^2t}.
\end{align}

\item \label{c02}For all 
$t_1,t_2\in [0,T]$ with $t_1<t_2$ we have that
\begin{align}
&
\left\lVert
X_\infty(t_1)
-X_\infty(t_2)
\right\rVert_{L^2(\P;\R^d)}\nonumber\\&\leq  \sqrt{t_2-t_1}\max \{\lVert\xi_d\rVert_{L^2(\P;\R^d)}, \lVert\mu(0,0)\rVert, \lVert\sigma(0,0)\rVert\}
3e^{3(\sqrt{T}+1)^2(T+1)c^2}.
\end{align}
\item \label{a10}For all $K\in \N$, $t\in [0,T]$ we have that
\begin{align}
&
\lVert X_K(t)-X_\infty(t)\rVert_{L^2(\P;\R^d)}
\nonumber\\&\leq 4.5
\sqrt{T/K}
\max \{\lVert\xi_d\rVert_{L^2(\P;\R^d)}, \lVert\mu(0,0)\rVert, \lVert\sigma(0,0)\rVert\}e^{5(\sqrt{T}+1)^2(T+1)c^2}.
\end{align}
\end{enumerate} 
\end{lemma}
\begin{proof}
[Proof of \cref{c12}]Recursion along the time grid shows \eqref{a04} for $K\in \N$. Next,
let $\mathcal{P}_2(\R^d)$ denote the set of probability measure on $\R^d$ with finite second moments. 
Let 
$\widetilde{\mu}\colon \R^d\times\mathcal{P}_2(\R^d)\to \R^d$,
$\widetilde{\sigma}\colon \R^d\times\mathcal{P}_2(\R^d)\to \R^{d\times d}$ satisfy for all $x\in \R^d$, $\nu\in \mathcal{P}(\R^d)$ that
$\widetilde{\mu}(x,\nu)= \int_{\R^d}\mu (x,y)\,\nu(dy)$,
$\widetilde{\sigma}(x,\nu)= \int_{\R^d}\sigma (x,y)\,\nu(dy)$. 
Following the arguments in the proof of \cite[Lemma~2.1]{NN2025} we see that the assumptions of \cite[Theorem~3.3]{RST2019} (with $b\gets \widetilde{\mu}$, $\sigma\gets\widetilde{\sigma}$) are satisfied. Therefore, \cite[Theorem~3.3]{RST2019} and Remark~\ref{x01}
show \eqref{a04} for $K=\infty$.

Next, \eqref{a04}, the triangle inequality,  and It\^o's isometry   show for all $t\in[0,T] $, $K\in \N\cup\{\infty\}$ that
\begin{align} 
&
\left\lVert X_K(t)\right\rVert_{L^2(\P,\R^d)}\nonumber\\
&\leq 
\lVert\xi_d\rVert+\int_{0}^{t}
\left\lVert\E \!\left[\mu(x,X_K(\rdown{s}_K))\right]\!|_{x=X_K(\rdown{s}_K)}\right\rVert_{L^2(\P;\R^d)}ds\nonumber\\
&\quad 
+\left\lVert
\int_{0}^{t}\E \!\left[\sigma(x,X_K(\rdown{s}_K))\right]\!|_{x=X_K(\rdown{s}_K)}\,dW(s)\right\rVert_{L^2(\P;\R^d)}\nonumber
\\
&\leq 
\lVert\xi_d\rVert+
\sqrt{T}
\left(\int_{0}^{t}
\left\lVert\E \!\left[\mu(x,X_K(\rdown{s}_K))\right]\!|_{x=X_K(\rdown{s}_K)}\right\rVert^2_{L^2(\P;\R^d)}ds\right)^\frac{1}{2}\nonumber\\
&\quad 
+\left(
\int_{0}^{t}
\left\lVert
\E \!\left[\sigma(x,X_K(\rdown{s}_K))\right]\!|_{x=X_K(\rdown{s}_K)}\right\rVert_{L^2(\P;\R^d)}^2 ds\right)^\frac{1}{2}
\nonumber\\
&\leq 
\lVert\xi_d\rVert+(\sqrt{T}+1)
\max\{\lVert\mu(0,0)\rVert,\lVert\sigma(0,0)\rVert\}
+(\sqrt{T}+1)c \left(\int_{0}^{t}
\lVert X_K(\rdown{s}_K)\rVert^2_{L^2(\P;\R^d)}ds
\right)^\frac{1}{2}.
\end{align}
This and Gr\"onwall's lemma (cf., e.g., \cite[Corollary~2.2]{HN2022})  show for all $t\in [0,T]$, $K\in \N\cup\{\infty\}$ that
\begin{align}
\left\lVert X_K(t)\right\rVert_{L^2(\P,\R^d)}
\leq \sqrt{2}\left(
\lVert\xi_d\rVert+(\sqrt{T}+1)
\max\{\lVert\mu(0,0)\rVert,\lVert\sigma(0,0)\rVert\}
\right)e^{(\sqrt{T}+1)^2c^2t}.
\end{align}
This proves \eqref{a09}.

Next, \eqref{a09}, the fact that $c\geq 1$, the fact that $1+\sqrt{2}\leq 3$ show for all
$s\in [0,T]$ that
\begin{align}
&\max\{
\lVert
\E [\mu(x,X_\infty(s))]|_{x=X_\infty(s)}\rVert_{L^2(\P,\R^d)} ,\lVert\E [\sigma(x,X_\infty(s))]|_{x=X_\infty(s)} 
\rVert_{L^2(\P,\R^{d\times d})}
\}\nonumber\\
&\leq 
 \max \{\lVert \mu(0,0)\rVert, \lVert\sigma(0,0)\rVert\}+ c
 \lVert X_\infty(s)\rVert_{L^2(\P;\R^d)}\nonumber\\
&\leq 
\max
\{\lVert \mu(0,0)\rVert, \lVert\sigma(0,0)\rVert\}\nonumber\\&\quad 
+c
\sqrt{2}\left(
\lVert\xi_d\rVert+(\sqrt{T}+1)
\max\{\lVert\mu(0,0)\rVert,\lVert\sigma(0,0)\rVert\}
\right)e^{(\sqrt{T}+1)^2c^2T}
\nonumber\\
&\leq 3c\left(
\lVert\xi_d\rVert+(\sqrt{T}+1)
\max\{\lVert\mu(0,0)\rVert,\lVert\sigma(0,0)\rVert\}
\right)e^{(\sqrt{T}+1)^2c^2T}.
\end{align}
This, \eqref{a04}, the triangle inequality, Jensen's inequality, It\^o's isometry, and the fact that $c\geq 1$ show that for all 
$t_1,t_2\in [0,T]$ with $t_1<t_2$ we have that
\begin{align} 
&
\left\lVert
X_\infty(t_1)
-X_\infty(t_2)
\right\rVert_{L^2(\P;\R^d)}\nonumber
\\
&\leq 
\left\lVert\int_{t_1}^{t_2}
\E\!\left[
\mu(x,X_\infty(s))\right]|_{x=X_\infty(s)}\,ds\right\rVert_{L^2(\P;\R^d)}
+ \left\lVert
\int_{t_1}^{t_2}
\E \!\left[
\sigma(x,X_\infty(s))]|_{x=X_\infty(s)}\right]dW(s)
\right\rVert_{L^2(\P;\R^d)}\nonumber\\
&\leq
\sqrt{T}\left(
\int_{t_1}^{t_2}
\left\lVert
\E\!\left[
\mu(x,X_\infty(s))\right]|_{x=X_\infty(s)}\right\rVert_{L^2(\P;\R^d)}^2ds\right)^\frac{1}{2}\nonumber
\\&\quad
+ \left(
\int_{t_1}^{t_2}
\left\lVert
\E \!\left[
\sigma(x,X_\infty(s))]|_{x=X_\infty(s)}\right]\right\rVert_{L^2(\P;\R^d)}^2ds\right)^\frac{1}{2}\nonumber\\
&\leq (\sqrt{T}+1)\sqrt{t_2-t_1}3c\left(
\lVert\xi_d\rVert+(\sqrt{T}+1)
\max\{\lVert\mu(0,0)\rVert,\lVert\sigma(0,0)\rVert\}
\right)e^{(\sqrt{T}+1)^2c^2T}\nonumber\\
&\leq \sqrt{t_2-t_1}
\left(
\lVert\xi_d\rVert+(\sqrt{T}+1)
\max\{\lVert\mu(0,0)\rVert,\lVert\sigma(0,0)\rVert\}
\right)3e^{(\sqrt{T}+1)^2c^2(T+1)}
\nonumber\\
&\leq 
 \sqrt{t_2-t_1}2(\sqrt{T}+1)\max \{\lVert\xi_d\rVert_{L^2(\P;\R^d)}, \lVert\mu(0,0)\rVert, \lVert\sigma(0,0)\rVert\}
3e^{(\sqrt{T}+1)^2(T+1)c^2}\nonumber\\
&\leq 
 \sqrt{t_2-t_1}\max \{\lVert\xi_d\rVert_{L^2(\P;\R^d)}, \lVert\mu(0,0)\rVert, \lVert\sigma(0,0)\rVert\}
3e^{3(\sqrt{T}+1)^2(T+1)c^2}.\label{c05}
\end{align}
This shows \eqref{c02}.

Next, \eqref{a11} and the triangle inequality prove for all $s\in[0,T] $, $K\in \N$ that
\begin{align}
&
\left\lVert
\E[\mu(x,X_\infty(s))]|_{x=X_\infty(s)}
-
\E[\mu(y,X_K(\rdown{s}_K))]|_{y=X_K(\rdown{s}_K)}\right\rVert_{L^2(\P;\R^d)}\nonumber
\\
&
\leq 
\left\lVert
\left\lVert
\mu(x,X_\infty(s))-\mu(y,X_K(\rdown{s}_K))
\right\rVert_{L^2(\P;\R^d)}|_{(x,y)=((X_\infty(s)),X_K(\rdown{s}_K))}\right\rVert_{L^2(\P;\R)}\nonumber\\
&
\leq 
\left\lVert
\left[0.5c
\lVert x-y\rVert+ 0.5c\left\lVert
X_\infty(s))-X_K(\rdown{s}_K)
\right\rVert_{L^2(\P;\R^d)}
\right]\!|_{(x,y)=((X_\infty(s)),X_K(\rdown{s}_K))}\right\rVert_{L^2(\P;\R)}\nonumber\\
&\leq c 
\left\lVert
X_\infty(s)-
X_K(\rdown{s}_K)
\right\rVert_{L^2(\P;\R^d)}\nonumber\\
&\leq c\left\lVert X_\infty(s)-X_\infty(\rdown{s}_K)\right\rVert_{L^2(\P;\R^d)}
+
c\left\lVert X_\infty(\rdown{s}_K)-
X_K(\rdown{s}_K)
\right\rVert_{L^2(\P;\R^d)}
\end{align}
and similarly that
\begin{align} \begin{split} 
&
\left\lVert
\E[\sigma(x,X_\infty(s))]|_{x=X_\infty(s)}
-
\E[\sigma(y,X_K(\rdown{s}_K))]|_{y=X_K(\rdown{s}_K)}\right\rVert_{L^2(\P;\R^{d\times d})}
\\
&\leq c\left\lVert X_\infty(s)-X_\infty(\rdown{s}_K)\right\rVert_{L^2(\P;\R^d)}
+
c\left\lVert X_\infty(\rdown{s}_K)-
X_K(\rdown{s}_K)
\right\rVert_{L^2(\P;\R^d)}.
\end{split}
\end{align}
This, the triangle inequality, Jensen's inequality, It\^o's isometry, \eqref{c05}, and the fact that $c\geq 1$
imply for all $K\in \N$, $t\in [0,T]$  
that 
\begin{align} 
&
\lVert X_K(t)-X_\infty(t)\rVert_{L^2(\P;\R^d)}\nonumber
\\
&
= 
\sqrt{T}\left(\int_{0}^{t}
\left\lVert
\E[\mu(x,X_\infty(s))]|_{x=X_\infty(s)}
-
\E[\mu(y,X_K(\rdown{s}_K))]|_{y=X_K(\rdown{s}_K)}\right\rVert_{L^2(\P;\R^d)}^2ds
\right)^\frac{1}{2}\nonumber\\
&\quad 
+
\left(\int_{0}^{t}
\left\lVert
\E[\sigma(x,X_\infty(s))]|_{x=X_\infty(s)}
-
\E[\sigma(y,X_K(\rdown{s}_K))]|_{y=X_K(\rdown{s}_K)}\right\rVert_{L^2(\P;\R^d)}^2ds
\right)^\frac{1}{2}\nonumber\\
&\leq (\sqrt{T}+1)\Biggl[
c
\left(\int_{0}^{t}
\left\lVert X_\infty(s)-X_\infty(\rdown{s}_K)\right\rVert_{L^2(\P;\R^d)}^2ds
\right)^\frac{1}{2}\nonumber\\
&\qquad \qquad\qquad
+
c\left(\int_{0}^{t}\left\lVert X_\infty(\rdown{s}_K)-
X_K(\rdown{s}_K)
\right\rVert_{L^2(\P;\R^d)}^2ds\right)^\frac{1}{2}
\Biggr]\nonumber\\
&\leq c\sqrt{T}(\sqrt{T}+1)\sqrt{T/K}\max \{\lVert\xi_d\rVert_{L^2(\P;\R^d)}, \lVert\mu(0,0)\rVert, \lVert\sigma(0,0)\rVert\}
3e^{3(\sqrt{T}+1)^2(T+1)c^2}\nonumber\\
&\quad
+c(\sqrt{T}+1)
\left(\int_{0}^{t}\left\lVert X_\infty(\rdown{s}_K)-
X_K(\rdown{s}_K)
\right\rVert_{L^2(\P;\R^d)}^2ds\right)^\frac{1}{2}\nonumber\\
&\leq \sqrt{T/K}
\max \{\lVert\xi_d\rVert_{L^2(\P;\R^d)}, \lVert\mu(0,0)\rVert, \lVert\sigma(0,0)\rVert\}
3e^{4(\sqrt{T}+1)^2(T+1)c^2}\nonumber\\&\quad +c(\sqrt{T}+1)
\left(\int_{0}^{t}\left\lVert X_\infty(\rdown{s}_K)-
X_K(\rdown{s}_K)
\right\rVert_{L^2(\P;\R^d)}^2ds\right)^\frac{1}{2}.
\end{align}
This and Gr\"onwall's inequality (cf., e.g., \cite[Corollary~2.2]{HN2022})   imply for all 
$K\in \N$, $t\in [0,T]$ that
\begin{align} 
&
\lVert X_K(t)-X(t)\rVert_{L^2(\P;\R^d)}\nonumber
\\
&\leq \sqrt{2}\sqrt{T/K}
\max \{\lVert\xi_d\rVert_{L^2(\P;\R^d)}, \lVert\mu(0,0)\rVert, \lVert\sigma(0,0)\rVert\}
3e^{4(\sqrt{T}+1)^2(T+1)c^2}e^{c^2(\sqrt{T}+1)^2T}\nonumber\\
&\leq 4.5
\sqrt{T/K}
\max \{\lVert\xi_d\rVert_{L^2(\P;\R^d)}, \lVert\mu(0,0)\rVert, \lVert\sigma(0,0)\rVert\}e^{5(\sqrt{T}+1)^2(T+1)c^2}
\end{align}
This shows \eqref{a10} and completes the proof of \cref{c12}.
\end{proof}
\section{Proof of Theorem~\ref{c01}}
\label{SecMLP}

In this section, we provide the proof of Theorem~\ref{c01}, which is divided into several steps. First, in Step~\ref{st02} we apply \cref{c12} to establish the existence, uniqueness, and upper bound of the exact solutions and the discretization error. In Step~\ref{s01} we establish measurability and distributional properties. In Step~\ref{s06} we prove the square-integrability of the approximations. In Step~\ref{s10} we consider the bias. In Step~\ref{s11} we consider the statistical error. In Step~\ref{s09} we obtain the global error. In Step~\ref{s08} we analyze the computational cost. In Step~\ref{s07} we analyze the computational complexity.

Throughout the proof we use the following notations. To lighten the notation, we denote by $\E _\mathcal{G}[\cdot]$ and $\var_\mathcal{G}[\cdot]$ the conditional expectation $\E [\cdot|\mathcal{G}]$ and the conditional variance
$\var [\cdot|\mathcal{G}]$, respectively, i.e., we have for all $d\in \N$ and all random variables $X\colon \Omega\to\R^d$ that
$\var_\mathcal{G}[X]= \E_\mathcal{G}\! \left[\left\lVert X- \E_{\mathcal{G}} [X]\right\rVert^2\right]
$. Next, for every $d,n,m,K\in \N$, let $\mathcal{G}^d_{n,m,K}$ the sigma-algebra generated by
$W^{d,0}$, $ X^{d,0}_{\ell,m,K} $, $\ell\in [0,n]\cap\Z$. Moreover, when applying a result we often use a phrase like ``Lemma 3.8 with $d\gets (d-1)$'' that should be read as ``Lemma 3.8 applied with $d$ (in the notation of Lemma 3.8) replaced by $(d-1 )$ (in the current notation)''.

\begin{proof}
[Proof of \cref{c01}]\renewcommand{\emph}[1]{\underline{\textit{#1}}}
We divide this proof into several steps. 

\newcounter{step}
\refstepcounter{step}\label{st02} 
\emph{Step~\thestep.} We establish the existence, uniqueness, and upper bound of the exact solutions and the discretization error. Observe that
\cref{c12} (applied for every $\theta\in \Theta$, $d\in \N$ with $T\gets T$, $c\gets c$, $d\gets d$,
$
\xi\gets \xi_d$, $\mu\gets \mu_d$, $\sigma\gets \sigma_d$,
$W\gets W^{d,\theta}$ in the notation of \cref{c12}) and \eqref{a11b} show that the following items are true.
\begin{enumerate}[(A)]
\item\label{a04b} 
For all $K\in \N\cup\{\infty\}$, $\theta\in \Theta$, $d\in \N$ there exists
a unique $\R^d$-valued $(\F_t)_{t\in [0,T]}$-adapted process $X^{d,\theta}_K=(X^{d,\theta}_K(t))_{t\in [0,T]}$ with continuous sample paths such that we have $\P$-a.s.\ for all $t\in [0,T]$ that $
\sup_{s\in [0,T]}\lVert X_K^{d,\theta}(s)\rVert_{L^2(\P,\R^d)}<\infty
$ and
\begin{align}
X_K^{d,\theta}(t)&=\xi_d+\int_{0}^{t}\E \!\left[\mu(x,X^{d,\theta}_K(\rdown{s}_K))\right]\!|_{x=X^{d,\theta}_K(\rdown{s}_K)}\,ds\nonumber\\
&\quad 
+
\int_{0}^{t}\E \!\left[\sigma(x,X^{d,\theta}_K(\rdown{s}_K))\right]\!|_{x=X^{d,\theta}_K(\rdown{s}_K)}\,dW^{d,\theta}(s).\label{a14}
\end{align}
\item\label{a09b} For all $t\in [0,T]$, $K\in \N\cup\{\infty\}$, $\theta\in \Theta$, $d\in \N$ we have that
\begin{align}
\left\lVert X^{d,\theta}_K(t)\right\rVert_{L^2(\P,\R^d)}
\leq \sqrt{2}\left(
\lVert\xi_d\rVert+(\sqrt{T}+1)
\max\{\lVert\mu_d(0,0)\rVert,\lVert\sigma_d(0,0)\rVert\}
\right)e^{(\sqrt{T}+1)^2c^2t}.\label{a18}
\end{align}
\item 
For all $K\in \N$, $t\in [0,T]$, $d\in \N$, $\theta\in \Theta$ we have that
\begin{align}
&\left
\lVert X^{d,\theta}_K(t)-X^{d,\theta}_\infty(t)\right\rVert_{L^2(\P;\R^d)}\nonumber\\
&
\leq 4.5
\sqrt{T/K}
\max \{\lVert\xi_d\rVert_{L^2(\P;\R^d)}, \lVert\mu_d(0,0)\rVert, \lVert\sigma_d(0,0)\rVert\}e^{4(\sqrt{T}+1)(T+1)c^2}.
\label{a15}
\end{align}
\end{enumerate}

\refstepcounter{step}
\emph{Step~\thestep.}\label{s01}
We establish measurability and distributional properties. 
First, the assumptions on measurability, basic facts on measurable functions,   \eqref{a05c}, induction, and the fact that
 $\forall\,d,K,m\in\N,\theta\in\Theta\colon 
X_{0,m,K}^{d,\theta}=0
$ prove for all $d,K,m\in \N$, $n\in\N_0$, $\theta\in\Theta$ that
${X}_{n,m,K}^{d,\theta}: (\Omega,\mathcal{F}) \rightarrow \R^d$ defined in \eqref{a05c} is measurable.
Next, the fact that
$\forall\,d,K,m\in\N,\theta\in\Theta\colon X_{0,m,K}^{d,\theta}=0$,
\eqref{a05c}, basic facts on sigma-algebras, and induction prove for all $n\in\N_0$,  $d,K,m\in\N$,
$\theta\in\Theta$ that
\begin{align}\begin{split}
&\sigma\!\left(\left\{W^{d,\theta}(t),X^{d,\theta}_{\ell,m,K}(t)
\colon \ell\in \{0,1,\ldots,n\},t\in [0,T]\right\}\right)
\\
&\subseteq  \sigma \!\left(\left\{W^{d,\theta}(t), W^{d,(\theta,i,\nu)}(t),\unif^{(\theta,i,\nu)}\colon i\in \{0,1,\ldots,n\},\nu\in\Theta, t\in[0,T] \right\} \right).
\end{split}\label{a01v}\end{align}
This and the fact that
$\forall\,d,K,m\in\N,\theta\in\Theta\colon 
X_{0,m,K}^{d,\theta}=0
$
 prove for all 
$n,m\in\N$,
$\theta\in\Theta$,
$k,\ell\in\N$,
$j\in\{\ell-1,\ell\}$ that
\begin{align}\begin{split}
&\sigma\!\left(\left\{X^{d,(\theta,n,k,\ell)}_{j,m,K}(t)
\colon t\in [0,T]\right\}\right)\\
&
\subseteq  \sigma \!\left(\left\{W^{d,(\theta,n,k,\ell)}(t), W^{d,(\theta,n,k,\ell,\nu)}(t),\unif^{(\theta,n,k,\ell,\nu)}\colon\nu\in\Theta, t\in[0,T] \right\} \right).
\end{split}\end{align}
This, \eqref{a01v}, and the independence assumptions show
for all 
$d,K,n,m\in\N$,
$\theta\in\Theta$
that
\begin{equation}
(W^{d,\theta},(X^{d,\theta}_{j,m,K})_{j\in [0,n-1]\cap\Z}),\quad
(X^{d,(\theta,n,k,\ell)}_{\ell,m,K}, X^{d,(\theta,n,k,\ell)}_{\ell-1,m,K}),\quad 
\unif^{(\theta,n,k,\ell)},\quad 
k,\ell\in\N,
\end{equation}
are independent.
This, the fact that
$\forall\,d,K,m\in\N,\theta\in\Theta\colon 
X_{0,m,K}^{d,\theta}=0
$, \eqref{a05c},
the disintegration theorem 
(see, e.g., \cite[Lemma 2.2]{HJK+2020}), and induction
show 
for all 
$m\in\N$,
$n\in\N_0$  that
$
( W^{d,\theta},(X_{\ell,m,K}^{d,\theta})_{ \ell\in\{0,1,\ldots,n\}}
)$, $\theta\in\Theta$, are identically distributed.

\refstepcounter{step}
\emph{Step~\thestep.} \label{s06}We establish that the approximations are square-integrable. The triangle inequality, \eqref{a11b}, the distributional properties (see Step~\ref{s01}), and the disintegration theorem (see, e.g., \cite[Lemma 2.2]{HJK+2020})
prove for all $d,m,n,K\in \N$, $\ell \in [1,n-1]\cap\Z$, $j\in \{\ell-1,\ell\}$, $t\in [0,T]$ that
\begin{align}
&
\left\lVert
t
\mu_d (X_{j,m,K}^{d,\theta} (\rdown{t\unif^{(\theta,n,k,\ell)}}_K ),X_{j,m,K}^{d,(\theta,n,k,\ell)} (\rdown{t\unif^{(\theta,n,k,\ell)}}_K ))\right\rVert_{L^2(\P;\R^d)}\nonumber\\
&
=\sqrt{t}\left(\int_{0}^{t}
\left\lVert
\mu_d (X_{j,m,K}^{d,\theta} (\rdown{s}_K ),X_{j,m,K}^{d,(\theta,n,k,\ell)} (\rdown{s}_K ))
\right\rVert_{L^2(\P;\R^d)}^2
ds\right)^\frac{1}{2}\nonumber\\
&
\leq 
T\left\lVert\mu_d(0,0)\right\rVert+
c\sqrt{T}
\left(\int_{0}^{T}
\left
\lVert X^{d,0}_{j,m,K}(\rdown{s}_K) )
\right\rVert_{L^2(\P;\R^d)}^2 ds\right)^\frac{1}{2}\nonumber\\
&\leq 
T\left\lVert\mu_d(0,0)\right\rVert+
cT\sup_{s\in [0,T]} 
\left
\lVert X^{d,0}_{j,m,K}(s) )
\right\rVert_{L^2(\P;\R^d)}.\label{r02}
\end{align}
Next, It\^o's isometry, the triangle inequality, \eqref{a11b}, the distributional properties (see Step~\ref{s01}), and the disintegration theorem (see, e.g., \cite[Lemma 2.2]{HJK+2020}) imply
for all $d,m,n,K\in \N$, $\ell \in [1,n-1]\cap\Z$, $j\in \{\ell-1,\ell\}$, $t\in [0,T]$ that
\begin{align}
&\left\lVert
\int_{0}^{t}
\sigma_d (X_{j,m,K}^{d,\theta} (\rdown{s}_K ),X_{j,m,K}^{d,(\theta,n,k,\ell)} (\rdown{s}_K ))\,
dW^{d,\theta}(s)
\right\rVert_{L^2(\P;\R^d)}
\nonumber\\
&=\left(\int_{0}^{t}
\left\lVert
\sigma_d(X_{j,m,K}^{d,\theta} (\rdown{s}_K ),X_{j,m,K}^{d,(\theta,n,k,\ell)} (\rdown{s}_K ))
\right\rVert_{L^2(\P;\R^{d\times d})}^2
ds\right)^\frac{1}{2}\nonumber\\
&
\leq 
\sqrt{T}\left\lVert\sigma_d(0,0)\right\rVert+
c
\left(\int_{0}^{T}
\left
\lVert X^{d,0}_{j,m,K}(\rdown{s}_K) )
\right\rVert_{L^2(\P;\R^d)}^2 ds\right)^\frac{1}{2}\nonumber\\
&\leq 
\sqrt{T}\left\lVert\sigma_d(0,0)\right\rVert+
c\sqrt{T}\sup_{s\in [0,T]} 
\left
\lVert X^{d,0}_{j,m,K}(s) )
\right\rVert_{L^2(\P;\R^d)}.\label{r03}
\end{align}
This, the fact that
$\forall\, d,k,m\in \N, \theta\in \Theta\colon X^{d,\theta}_{0,m,K}=0$, \eqref{a05c}, the triangle inequality, and induction show for all
$d,n,m,K\in \N$, $\ell\in [1,n-1]\cap\Z$, $j\in \{\ell-1,\ell\}$
that
\begin{align}\small\begin{split}
&
\sup_{t\in[0,T]}\Biggl[
\left\lVert 
X_{n,m,K}^0(t)\right\rVert^2_{L^2(\P;\R^d)}+
\left\lVert 
\mu\bigl(X^0_{j,m,K}(t \unif^{(0,n,k,\ell)}),
X^{(0,n,k,\ell)}_{j,m,K}(t \unif^{(0,n,k,\ell)})
\bigr)\right\rVert_{L^2(\P;\R^d)}
\\
&\qquad\qquad\qquad +
\left\lVert
\int_{0}^{t}
\sigma_d (X_{j,m,K}^{d,\theta} (\rdown{s}_K ),X_{j,m,K}^{d,(\theta,n,k,\ell)} (\rdown{s}_K ))\,
dW^{d,\theta}(s)
\right\rVert_{L^2(\P;\R^d)}
\Biggr]<\infty.\end{split}
\end{align}

\refstepcounter{step}
\emph{Step~\thestep.}\label{s10}
We analyze for all $d,n,m,K\in \N$ the bias of $X^{d,0}_{n,m,K}$.
Recall for every $d,n,m,K\in \N$ that $\mathcal{G}^d_{n,m,K}$ is the sigma-algebra generated by
$W^{d,0}$, $ X^{d,0}_{\ell,m,K} $, $\ell\in [0,n]\cap\Z$.
Integrability (see Step~\ref{s06}), linearity of the conditional expectations, the distributional properties (see Step~\ref{s01}), the disintegration theorem (see, e.g., \cite[Lemma 2.2]{HJK+2020}), and the fact that 
$\forall m,d,K\in \N\colon X^{d,0}_{0,m,K}=0 $ show for all 
$d,m,n,K\in \N$, $t\in [0,T]$ that
\begin{align} 
&\E_{\mathcal{G}^d_{n-1,m,K}}\!\left[
X^{d,0}_{n,m,K}(t)\right]\nonumber\\
&= \xi_d +t\mu_d(0,0)+\sigma_d(0,0)W^{d,0}(t)\nonumber\\
&\qquad+ \sum_{\ell=1}^{n-1}\sum_{k=1}^{m^{n-\ell}}
\frac{t}{m^{n-\ell}}
\E_{\mathcal{G}^d_{n-1,m,K}}
\Bigl[
\mu_d (X_{\ell,m,K}^{d,0} (\rdown{t\unif^{(0,n,k,\ell)}}_K ),X_{\ell,m,K}^{d,(0,n,k,\ell)} (\rdown{t\unif^{(0,n,k,\ell)}}_K ))\nonumber\\
&\qquad\qquad\qquad\qquad
-
\mu_d (X_{\ell-1,m,K}^{d,0} (\rdown{t\unif^{(0,n,k,\ell)}}_K ),X_{\ell-1,m,K}^{d,(0,n,k,\ell)} (\rdown{t\unif^{(0,n,k,\ell)}}_K ))
\Bigr]\nonumber
\\
&\qquad+ \sum_{\ell=1}^{n-1}\sum_{k=1}^{m^{n-\ell}}\int_{0}^{t}
\frac{1}{m^{n-\ell}}\E_{\mathcal{G}^d_{n-1,m,K}}
\Bigl[
\sigma_d(X_{\ell,m,K}^{d,0} (\rdown{s}_K ),X_{\ell,m,K}^{d,(0,n,k,\ell)} (\rdown{s}_K ))\nonumber\\
&\qquad\qquad\qquad\qquad\qquad
-
\sigma_d(X_{\ell-1,m,K}^{d,0} (\rdown{s}_K ),X_{\ell-1,m,K}^{d,(0,n,k,\ell)} (\rdown{s}_K ))\Bigr]
dW^{d,0}(s)\nonumber\\
&= \xi_d +t\mu_d(0,0)+\sigma_d(0,0)W^{d,0}(t)\nonumber\\
&\qquad+ \sum_{\ell=1}^{n-1}
t
\E_{\mathcal{G}^d_{n-1,m,K}}
\Bigl[
\mu_d (X_{\ell,m,K}^{d,0} (\rdown{t\unif^{1}}_K ),X_{\ell,m,K}^{d,1} (\rdown{t\unif^{1}}_K ))\nonumber\\
&\qquad\qquad\qquad\qquad
-
\mu_d (X_{\ell-1,m,K}^{d,0} (\rdown{t\unif^{1}}_K ),X_{\ell-1,m,K}^{d,1} (\rdown{t\unif^{1}}_K ))
\Bigr]\nonumber
\\
&\qquad+ \sum_{\ell=1}^{n-1}\int_{0}^{t}
\E_{\mathcal{G}^d_{n-1,m,K}}
\Bigl[
\sigma_d(X_{\ell,m,K}^{d,0} (\rdown{s}_K ),X_{\ell,m,K}^{d,1} (\rdown{s}_K ))\nonumber\\
&\qquad\qquad\qquad\qquad\qquad
-
\sigma_d(X_{\ell-1,m,K}^{d,0} (\rdown{s}_K ),X_{\ell-1,m,K}^{d,1} (\rdown{s}_K ))\Bigr]
dW^{d,0}(s)\nonumber\\
&=
 \xi_d +t\mu_d(0,0)+\sigma_d(0,0)W^{d,0}(t)\nonumber\\
&\qquad+ \sum_{\ell=1}^{n-1}
\int_{0}^{t}
\E_{\mathcal{G}^d_{n-1,m,K}}
\Bigl[
\mu_d (X_{\ell,m,K}^{d,0} (\rdown{s}_K ),X_{\ell,m,K}^{d,1} (\rdown{s}_K ))\nonumber\\
&\qquad\qquad\qquad\qquad
-
\mu_d (X_{\ell-1,m,K}^{d,0} (\rdown{s}_K ),X_{\ell-1,m,K}^{d,1} (\rdown{s}_K ))
\Bigr]ds\nonumber
\\
&\qquad+ \sum_{\ell=1}^{n-1}\int_{0}^{t}
\E_{\mathcal{G}^d_{n-1,m,K}}
\Bigl[
\sigma_d(X_{\ell,m,K}^{d,0} (\rdown{s}_K ),X_{\ell,m,K}^{d,1} (\rdown{s}_K ))\nonumber\\
&\qquad\qquad\qquad\qquad\qquad
-
\sigma_d(X_{\ell-1,m,K}^{d,0} (\rdown{s}_K ),X_{\ell-1,m,K}^{d,1} (\rdown{s}_K ))\Bigr]
dW^{d,0}(s)\nonumber\\
&=\xi_d+ \int_{0}^{t}
\E\!\left[\mu_d (x,X^{d,0}_{n-1,m,K}(\rdown{s}_K))\right]\!|_{x=X^{d,0}_{n-1,m,K}(\rdown{s}_K)}\,ds\nonumber\\
&\qquad
+
\int_{0}^{t}
\E\!\left[\sigma_d (x,X^{d,0}_{n-1,m,K}(\rdown{s}_K))\right]\!|_{x=X^{d,0}_{n-1,m,K}(\rdown{s}_K)}\,dW^{d,0}(s).
\end{align}
This and \eqref{a14} 
prove for all $d,m,n,K\in \N$, $t\in [0,T]$ that
\begin{align}
&
\E_{\mathcal{G}^d_{n-1,m,K}}\!\left[
X^{d,0}_{n,m,K}(t)\right]- X_K^{d,0}(t)\nonumber\\
&= \int_{0}^{t}
\E\!\left[\mu_d (x,X^{d,0}_{n-1,m,K}(\rdown{s}_K))\right]\!|_{x=X^{d,0}_{n-1,m,K}(\rdown{s}_K)}
-\E \!\left[\mu(x,X^{d,0}_K(\rdown{s}_K))\right]\!|_{x=X^{d,0}_K(\rdown{s}_K)}
\,ds\nonumber\\
&
+
\int_{0}^{t}
\E\!\left[\sigma_d (x,X^{d,0}_{n-1,m,K}(\rdown{s}_K))\right]\!|_{x=X^{d,0}_{n-1,m,K}(\rdown{s}_K)}-\E \!\left[\sigma(x,X^{d,0}_K(\rdown{s}_K))\right]\!|_{x=X^{d,0}_K(\rdown{s}_K)}\,dW^{d,0}(s).
\end{align}
This, the triangle inequality, Jensen's inequality, It\^o's isometry, \eqref{a11b}, and the disintegration theorem (see, e.g., \cite[Lemma 2.2]{HJK+2020}) show for all $d,m,n,K\in \N$, $t\in [0,T]$ that
\begin{align}
&
\left\lVert
\E_{\mathcal{G}^d_{n-1,m,K}}\!\left[
X^{d,0}_{n,m,K}(t)\right]- X_K^{d,0}(t)\right\rVert_{L^2(\P;\R^d)}
\nonumber\\
&\leq \sqrt{T}\Biggl(
\int_{0}^{t}
\Bigl\lVert
\E\!\left[\mu_d (x,X^{d,0}_{n-1,m,K}(\rdown{s}_K))\right]\!|_{x=X^{d,0}_{n-1,m,K}(\rdown{s}_K)}\nonumber\\
&\qquad\qquad
-\E \!\left[\mu_d(x,X^{d,0}_K(\rdown{s}_K))\right]\!|_{x=X^{d,0}_K(\rdown{s}_K)}\Bigr\rVert_{L^2(\P;\R^d)}^2
\,ds\Biggr)^\frac{1}{2}\nonumber\\
&
+\Biggl(
\int_{0}^{t}
\Bigl\lVert
\E\!\left[\sigma_d (x,X^{d,0}_{n-1,m,K}(\rdown{s}_K))\right]\!|_{x=X^{d,0}_{n-1,m,K}(\rdown{s}_K)}\nonumber\\
&\qquad\qquad-\E \!\left[\sigma_d(x,X^{d,0}_K(\rdown{s}_K))\right]\!|_{x=X^{d,0}_K(\rdown{s}_K)}\Bigr \rVert^2_{L^2(\P;\R^{d\times d})}ds
\Biggr)^\frac{1}{2}
\nonumber\\
&\leq (\sqrt{T}+1)c \left(\int_{0}^{t}\left\lVert X^{d,0}_{n-1,m,K}(\rdown{s}_K)-X^{d,0}_{K}(\rdown{s}_K)\right\rVert_{L^2(\P;\R^d)}^2 ds\right)^\frac{1}{2}.\label{b01}
\end{align}

\refstepcounter{step}\label{s11}
\emph{Step~\thestep.}\ 
We consider the statistical error. 
The distributional properties (cf.\ Step~\ref{s01}) imply  for all 
$d,K,n,m\in\N$, $\ell\in [1,n-1]\cap\N$ that
\begin{enumerate}[(A)]
\item  we have that
$(X^{d,(0,n,1,\ell)}_{\ell,m,K}, X^{d,(0,n,1,\ell)}_{\ell-1,m,K}, \unif^{(0,n,1,\ell)}) $ and 
$\mathcal{G}^d_{n-1,m,K}$ are independent,
\item 
we have $\P$-a.s.\ that
$(X^{d,(0,n,k,\ell)}_{\ell,m,K}, X^{d,(0,n,k,\ell)}_{\ell-1,m,K}, \unif^{(0,n,k,\ell)}) $, $k\in\N$,
are i.i.d. under $\P[\cdot |\mathcal{G}^d_{n-1,m,K}]$, and
\item 
we have  that
$((X^{d,(0,n,1,\ell)}_{\ell,m,K}, X^{d,(0,n,1,\ell)}_{\ell-1,m,K}) $
and 
$((X^{d,0}_{\ell,m,K}, X^{d,0}_{\ell-1,m,K}) $ are identically distributed.
\end{enumerate}
This,
\eqref{a05c},
the triangle inequality,
Bienaym\'e's identity,
the assumptions on distributions,  and 
the disintegration theorem (see, e.g., \cite[Lemma 2.2]{HJK+2020})
prove that for all 
$d,K,n,m\in\N$, 
$t\in [0,T]$ we have $\P$-a.s.\
that
\begin{align}
&\left(\var_{\mathcal{G}^d_{n-1,m,K}}\!\left
[
X^{d,0}_{n,m,K}(t)\right]\right)^\frac{1}{2} 
\nonumber\\
&= \sum_{\ell=1}^{n-1}\Biggl(\var_{\mathcal{G}^d_{n-1,m,K}}\Biggl[ \sum_{k=1}^{m^{n-\ell}}
\frac{t}{m^{n-\ell}}
\Bigl[
\mu_d (X_{\ell,m,K}^{d,0} (\rdown{t\unif^{(0,n,k,\ell)}}_K ),X_{\ell,m,K}^{d,(0,n,k,\ell)} (\rdown{t\unif^{(0,n,k,\ell)}}_K ))\nonumber\\
&\qquad\qquad\qquad\qquad
-
\mu_d (X_{\ell-1,m,K}^{d,0} (\rdown{t\unif^{(0,n,k,\ell)}}_K ),X_{\ell-1,m,K}^{d,(0,n,k,\ell)} (\rdown{t\unif^{(0,n,k,\ell)}}_K ))
\Bigr]\Biggr]\Biggr)^\frac{1}{2}\nonumber
\\
&\quad + \sum_{\ell=1}^{n-1}\Biggl(\var_{\mathcal{G}^d_{n-1,m,K}}\Biggl[\sum_{k=1}^{m^{n-\ell}}\int_{0}^{t}
\frac{1}{m^{n-\ell}}
\Bigl[
\sigma_d(X_{\ell,m,K}^{d,0} (\rdown{s}_K ),X_{\ell,m,K}^{d,(0,n,k,\ell)} (\rdown{s}_K ))\nonumber\\
&\qquad\qquad\qquad\qquad\qquad
-
\sigma_d(X_{\ell-1,m,K}^{d,0} (\rdown{s}_K ),X_{\ell-1,m,K}^{d,(0,n,k,\ell)} (\rdown{s}_K ))\Bigr]
dW^{d,0}(s)\Biggr]\Biggr)^\frac{1}{2}\nonumber\\
&=\sum_{\ell=1}^{n-1}\Biggl[\frac{1}{\sqrt{m^{n-\ell}}}
\biggl(\var_{\mathcal{G}^d_{n-1,m,K}}
\Bigl[
t\mu_d (X_{\ell,m,K}^{d,0} (\rdown{t\unif^{1}}_K ),X_{\ell,m,K}^{d,1} (\rdown{t\unif^{1}}_K ))\nonumber\\
&\qquad\qquad\qquad\qquad
-
t\mu_d (X_{\ell-1,m,K}^{d,0} (\rdown{t\unif^{1}}_K ),X_{\ell-1,m,K}^{d,1} (\rdown{t\unif^{1}}_K ))
\Bigr]\biggr)^\frac{1}{2}\Biggr]\nonumber\\
&\quad +\sum_{\ell=1}^{n-1}\Biggl[\frac{1}{\sqrt{m^{n-\ell}}}
\biggl(\var_{\mathcal{G}^d_{n-1,m,K}}\biggl[\int_{0}^{t}
\sigma_d (X_{\ell,m,K}^{d,0} (\rdown{s}_K ),X_{\ell,m,K}^{d,1} (\rdown{s}_K ))\nonumber
\\
&\qquad\qquad\qquad
-
\sigma_d (X_{\ell-1,m,K}^{d,0} (\rdown{s}_K ),X_{\ell-1,m,K}^{d,1} (\rdown{s}_K ))
\,dW^{d,0}(s)
\biggr]\biggr)^\frac{1}{2}\Biggr].\label{a16}
\end{align}
Next, a property of conditional expectations, 
the disintegration theorem (see, e.g., \cite[Lemma 2.2]{HJK+2020}), the distributional properties (see Step~\ref{s01}), 
the fact that for all $t\in [0,T]$ we have that 
$t\unif^1$ is uniformly distributed on $[0,t]$,
and \eqref{a11b} show for all 
$d,m,n,K\in \N$, $t\in [0,T]$, $\ell\in [1,n-1]$
that
\begin{align}
&\biggl \lVert\biggl(\var_{\mathcal{G}^d_{n-1,m,K}}
\Bigl[
t\mu_d (X_{\ell,m,K}^{d,0} (\rdown{t\unif^{1}}_K ),X_{\ell,m,K}^{d,1} (\rdown{t\unif^{1}}_K ))\nonumber\\
&\qquad\qquad\qquad\qquad
-
t\mu_d (X_{\ell-1,m,K}^{d,0} (\rdown{t\unif^{1}}_K ),X_{\ell-1,m,K}^{d,1} (\rdown{t\unif^{1}}_K ))
\Bigr]\biggr)^\frac{1}{2}\biggr \rVert_{L^2(\P;\R)}\nonumber
\\
&\leq \biggl \lVert\biggl(\E_{\mathcal{G}^d_{n-1,m,K}}
\Bigl[\Bigl \lVert
t\mu_d (X_{\ell,m,K}^{d,0} (\rdown{t\unif^{1}}_K ),X_{\ell,m,K}^{d,1} (\rdown{t\unif^{1}}_K ))\nonumber\\
&\qquad\qquad\qquad\qquad
-
t\mu_d (X_{\ell-1,m,K}^{d,0} (\rdown{t\unif^{1}}_K ),X_{\ell-1,m,K}^{d,1} (\rdown{t\unif^{1}}_K ))\Bigr \rVert^2
\Bigr]\biggr)^\frac{1}{2}\biggr \rVert_{L^2(\P;\R)}\nonumber\\
&\leq \Biggl( \E \biggl[\E_{\mathcal{G}^d_{n-1,m,K}}
\Bigl[\Bigl \lVert
t\mu_d (X_{\ell,m,K}^{d,0} (\rdown{t\unif^{1}}_K ),X_{\ell,m,K}^{d,1} (\rdown{t\unif^{1}}_K ))\nonumber\\
&\qquad\qquad\qquad\qquad
-
t\mu_d (X_{\ell-1,m,K}^{d,0} (\rdown{t\unif^{1}}_K ),X_{\ell-1,m,K}^{d,1} (\rdown{t\unif^{1}}_K ))\Bigr \rVert^2
\Bigr]\biggr]\Biggr)^\frac{1}{2}\nonumber\\
&=
 \Biggl( \E \biggl[\Bigl \lVert
t\mu_d (X_{\ell,m,K}^{d,0} (\rdown{t\unif^{1}}_K ),X_{\ell,m,K}^{d,1} (\rdown{t\unif^{1}}_K ))\nonumber\\
&\qquad\qquad\qquad\qquad
-
t\mu_d (X_{\ell-1,m,K}^{d,0} (\rdown{t\unif^{1}}_K ),X_{\ell-1,m,K}^{d,1} (\rdown{t\unif^{1}}_K ))\Bigr \rVert^2\biggr]\Biggr)^\frac{1}{2}\nonumber\\
&\leq 
\sqrt{T}
 \Biggl( \E \biggl[t\Bigl \lVert
\mu_d (X_{\ell,m,K}^{d,0} (\rdown{t\unif^{1}}_K ),X_{\ell,m,K}^{d,1} (\rdown{t\unif^{1}}_K ))\nonumber\\
&\qquad\qquad\qquad\qquad
-
\mu_d (X_{\ell-1,m,K}^{d,0} (\rdown{t\unif^{1}}_K ),X_{\ell-1,m,K}^{d,1} (\rdown{t\unif^{1}}_K ))\Bigr \rVert^2\biggr]\Biggr)^\frac{1}{2}\nonumber\\
&=\sqrt{T}
 \Biggl( \E \biggl[\int_{0}^{t}\Bigl \lVert
\mu_d (X_{\ell,m,K}^{d,0} (\rdown{s}_K ),X_{\ell,m,K}^{d,1} (\rdown{s}_K ))\nonumber\\
&\qquad\qquad\qquad\qquad
-
\mu_d (X_{\ell-1,m,K}^{d,0} (\rdown{s}_K ),X_{\ell-1,m,K}^{d,1} (\rdown{s}_K ))\Bigr \rVert^2ds\biggr]\Biggr)^\frac{1}{2}\nonumber\\
&\leq c\sqrt{T}\left(\E\!\left[\int_{0}^{t}\left\lVert  X^{d,0}_{\ell,m,K}(\rdown{s}_K)-
X^{d,0}_{\ell-1,m,K}(\rdown{s}_K)\right
\rVert^2ds\right] \right)^\frac{1}{2}\nonumber\\
&\leq c\sqrt{T}\sum_{j=\ell-1}^{\ell}\left(
\int_{0}^{t}\sup_{r\in [0,s]}\left\lVert X^{d,0}_{j,m,K}(r)-X^{d,0}_K(r)\right\rVert_{L^2(\P;\R^d)} ^2\,ds
\right)^\frac{1}{2}.
\end{align}
Similarly but using It\^o's formula we have for all
$d,K,n,m\in\N$, 
$t\in [0,T]$, $\ell\in [1,n-1]\cap\Z$ that, $\P$-a.s.,
\begin{align} 
&\Biggl\lVert
\biggl(\var_{\mathcal{G}^d_{n-1,m,K}}\biggl[\int_{0}^{t}
\sigma_d (X_{\ell,m,K}^{d,0} (\rdown{s}_K ),X_{\ell,m,K}^{d,1} (\rdown{s}_K ))\nonumber
\\
&\qquad\qquad\qquad
-
\sigma_d (X_{\ell-1,m,K}^{d,0} (\rdown{s}_K ),X_{\ell-1,m,K}^{d,1} (\rdown{s}_K ))
\,dW^{d,0}(s)
\biggr]\biggr)^\frac{1}{2}
\Biggr\rVert_{L^2(\P;\R)}\nonumber
\\
&
\leq \Biggl\lVert
\biggl(\E_{\mathcal{G}^d_{n-1,m,K}}\biggl[\biggl\lVert\int_{0}^{t}
\sigma_d (X_{\ell,m,K}^{d,0} (\rdown{s}_K ),X_{\ell,m,K}^{d,1} (\rdown{s}_K ))\nonumber\\
&\qquad\qquad\qquad\qquad
-
\sigma_d (X_{\ell-1,m,K}^{d,0} (\rdown{s}_K ),X_{\ell-1,m,K}^{d,1} (\rdown{s}_K ))
\,dW^{d,0}(s)\biggr\rVert^2
\biggr]\biggr)^\frac{1}{2}\Biggr\rVert_{L^2(\P;\R)}\nonumber\\
&= \Biggl(\E\biggl[
\E_{\mathcal{G}^d_{n-1,m,K}}\biggl[\biggl\lVert\int_{0}^{t}
\sigma_d (X_{\ell,m,K}^{d,0} (\rdown{s}_K ),X_{\ell,m,K}^{d,1} (\rdown{s}_K ))\nonumber\\
&\qquad\qquad\qquad\qquad
-
\sigma_d (X_{\ell-1,m,K}^{d,0} (\rdown{s}_K ),X_{\ell-1,m,K}^{d,1} (\rdown{s}_K ))
\,dW^{d,0}(s)\biggr\rVert^2\biggr]
\biggr]\Biggr)^\frac{1}{2}\nonumber\\
&\leq \Biggl(\E
\Biggl[\biggl\lVert\int_{0}^{t}
\sigma_d (X_{\ell,m,K}^{d,0} (\rdown{s}_K ),X_{\ell,m,K}^{d,1} (\rdown{s}_K ))\nonumber\\
&\qquad
\qquad
\qquad
-
\sigma_d (X_{\ell-1,m,K}^{d,0} (\rdown{s}_K ),X_{\ell-1,m,K}^{d,1} (\rdown{s}_K ))
\,dW^{d,0}(s)\biggr\rVert^2
\Biggr]\Biggr)^\frac{1}{2}\nonumber
\\
&=\biggl(
\int_{0}^{t}
\E\biggl[
\Bigl\lVert
\sigma_d (X_{\ell,m,K}^{d,0} (\rdown{s}_K ),X_{\ell,m,K}^{d,1} (\rdown{s}_K ))\nonumber\\
&\qquad\qquad\qquad\qquad
-
\sigma_d (X_{\ell-1,m,K}^{d,0} (\rdown{s}_K ),X_{\ell-1,m,K}^{d,1} (\rdown{s}_K ))
\Bigr\rVert^2
\biggr]ds\biggr)^\frac{1}{2}\nonumber
\\
&\leq c\left( \int_{0}^{t}
\left\lVert
X_{\ell,m,K}^{d,0} (\rdown{s}_K )
-X_{\ell-1,m,K}^{d,0} (\rdown{s}_K )
\right\rVert^2ds\right)^\frac{1}{2}
\nonumber\\
&\leq c\sum_{j=\ell-1}^{\ell}\left(
\int_{0}^{t}\sup_{r\in [0,s]}\left\lVert X^{d,0}_{j,m,K}(r)-X^{d,0}_K(r)\right\rVert_{L^2(\P;\R^d)} ^2\,ds
\right)^\frac{1}{2}.\label{a17}
\end{align}

\refstepcounter{step}\label{s09}
\emph{Step~\thestep.}\ We combine the bias, the statistical error, and the discretization error to obtain the global error and to show for all $d,m,n,K\in \N$, $t\in [0,T]$ that
\begin{align}
	&
	\left\lVert 
	X^{d,0}_{n,m,K}(t)-X_\infty^{d,0}(t)
	\right\rVert_{L^2(\P;\R^d)}\nonumber\\
	&\leq \sqrt{2}(\sqrt{T}+1)cd^ce^{(\sqrt{T}+1)^2c^2T}
	e^{2nc(\sqrt{T}+1)\sqrt{T} }\frac{e^{m/2}}{m^{n/2}}
	+4.5
	\sqrt{T/K}cd^ce^{4(\sqrt{T}+1)(T+1)c^2}.
	\label{a19c}
\end{align}
Combining \eqref{a16}--\eqref{a17}, properties of conditional expectations, and the triangle inequality we have for all $d,n,m,K\in \N$, $t\in [0,T]$ that
\begin{align}
&
\left\lVert 
X^{d,0}_{n,m,K}(t)-
\E _{\mathcal{G}^{d}_{n-1,m,K} }\!\left[
X^{d,0}_{n,m,K}(t)\right]
\right\rVert_{L^2(\P;\R^d)}^2\nonumber\\
&
=\E \!\left[
\left\lVert 
X^{d,0}_{n,m,K}(t)-
\E _{\mathcal{G}^{d}_{n-1,m,K} }\!\left[
X^{d,0}_{n,m,K}(t)\right]
\right\rVert^2\right]\nonumber\\
&=\E \left[\E_{\mathcal{G}^{d}_{n-1,m,K} } \!\left[
\left\lVert 
X^{d,0}_{n,m,K}(t)-
\E _{\mathcal{G}^{d}_{n-1,m,K} }\!\left[
X^{d,0}_{n,m,K}(t)\right]
\right\rVert^2\right]
\right]\nonumber\\
&=\E\!\left[\var_{\mathcal{G}^{d}_{n-1,m,K}\! }\left[X^{d,0}_{n,m,K}(t)\right]\right]\nonumber\\
&=
\left\lVert
\left(\var_{\mathcal{G}^d_{n-1,m,K}}\!\left
[
X^{d,0}_{n,m,K}(t)\right]\right)^\frac{1}{2} 
\right\rVert_{L^2(\P;\R)}
\nonumber\\
&\leq \sum_{\ell=1}^{n-1}
\frac{1}{\sqrt{m^{n-\ell}}}c(\sqrt{T}+1)\sum_{j=\ell-1}^{\ell}\left(
\int_{0}^{t}\sup_{r\in [0,s]}\left\lVert X^{d,0}_{j,m,K}(r)-X^{d,0}_K(r)\right\rVert_{L^2(\P;\R^d)} ^2\,ds
\right)^\frac{1}{2}\nonumber\\
&\leq \sum_{\ell=0}^{n-1}
\frac{(2-\1_{\{n-1\}}(\ell))c(\sqrt{T}+1)}{\sqrt{m^{n-\ell-1}}}
\left(
\int_{0}^{t}\sup_{r\in [0,s]}\left\lVert X^{d,0}_{\ell,m,K}(r)-X^{d,0}_K(r)\right\rVert_{L^2(\P;\R^d)} ^2\,ds
\right)^\frac{1}{2}.
\end{align}
This, the triangle inequality, and \eqref{b01} prove for all $d,n,m,K\in \N$, $t\in [0,T]$ that
\begin{align}
&
\left\lVert 
X^{d,0}_{n,m,K}(t)-X_K^{d,0}(t)
\right\rVert_{L^2(\P;\R^d)}\nonumber\\
&
\leq 
\left\lVert 
X^{d,0}_{n,m,K}(t)-
\E _{\mathcal{G}^{d}_{n-1,m,K} }\!\left[
X^{d,0}_{n,m,K}(t)\right]
\right\rVert_{L^2(\P;\R^d)}
+
\left\lVert 
\E _{\mathcal{G}^{d}_{n-1,m,K} }\!\left[
X^{d,0}_{n,m,K}(t)\right]
-X_K^{d,0}(t)
\right\rVert_{L^2(\P;\R^d)}\nonumber\\
&\leq \sum_{\ell=0}^{n-1}
\frac{2c(\sqrt{T}+1)}{\sqrt{m^{n-\ell-1}}}
\left(
\int_{0}^{t}\sup_{r\in [0,s]}\left\lVert X^{d,0}_{\ell,m,K}(r)-X^{d,0}_K(r)\right\rVert_{L^2(\P;\R^d)} ^2\,ds
\right)^\frac{1}{2}.
\end{align}
This and the monotonicity of the integral show
for all $d,n,m,K\in \N$, $t\in [0,T]$ that
\begin{align}
&\sup_{u\in[0,t]}
\left\lVert 
X^{d,0}_{n,m,K}(u)-X_K^{d,0}(u)
\right\rVert_{L^2(\P;\R^d)}\nonumber\\
&\leq \sum_{\ell=0}^{n-1}
\frac{2c(\sqrt{T}+1)}{\sqrt{m^{n-\ell-1}}}
\left(
\int_{0}^{t}\sup_{r\in [0,s]}\left\lVert X^{d,0}_{\ell,m,K}(r)-X^{d,0}_K(r)\right\rVert_{L^2(\P;\R^d)} ^2\,ds
\right)^\frac{1}{2}.
\end{align}
This, \cite[Lemma~3.9]{HJKN2020}, \eqref{a18}, the fact that
$\forall\, d,m,K\colon X^{d,0}_{0,m,K}=0$, and \eqref{a11c}
 show for all $n,m,d,K\in \N$,
$t\in [0,T]$ that
\begin{align}
&
\sup_{u\in[0,t]}
\left\lVert 
X^{d,0}_{n,m,K}(u)-X_K^{d,0}(u)
\right\rVert_{L^2(\P;\R^d)} \nonumber \\
&
\leq 
\left(2c(\sqrt{T}+1)T^{1/2} \sup_{t\in [0,T]} \left\lVert X^{d,0}_{K}(t)\right\rVert_{L^2(\P;\R^d)}  \right)
\left[\max_{k\in\{0,1,\ldots,n\}} \frac{1}{\sqrt{m^{n-k}k!}}\right]
\left(1+2c(\sqrt{T}+1)T^{1/2}\right)^{n-1}\nonumber\\
&\leq \left(
\sup_{t\in [0,T]} \left\lVert X^{d,0}_{K}(t)\right\rVert_{L^2(\P;\R^d)}
\right)\left(1+2c(\sqrt{T}+1)T^{1/2}\right)^{n}
\left[\max_{k\in\{0,1,\ldots,n\}} \frac{\sqrt{m^k}}{\sqrt{m^{n}k!}}\right]\nonumber\\
&\leq \sqrt{2}\left(
\lVert\xi_d\rVert+(\sqrt{T}+1)
\max\{\lVert\mu_d(0,0)\rVert,\lVert\sigma_d(0,0)\rVert\}
\right)e^{(\sqrt{T}+1)^2c^2T}
e^{2nc(\sqrt{T}+1)\sqrt{T} }\frac{e^{m/2}}{m^{n/2}}\nonumber\\
&\leq \sqrt{2}(\sqrt{T}+1)cd^ce^{(\sqrt{T}+1)^2c^2T}
e^{2nc(\sqrt{T}+1)\sqrt{T} }\frac{e^{m/2}}{m^{n/2}}.
\end{align}
This, the triangle inequality, and \eqref{a15} show for all $d,m,n,K\in \N$, $t\in [0,T]$ that
\begin{align}
&
\left\lVert 
X^{d,0}_{n,m,K}(t)-X_\infty^{d,0}(t)
\right\rVert_{L^2(\P;\R^d)}\nonumber\\
&
\leq 
\left\lVert 
X^{d,0}_{n,m,K}(t)-X_K^{d,0}(t)
\right\rVert_{L^2(\P;\R^d)}
+
\left\lVert 
X_K^{d,0}(t)-X_\infty^{d,0}(t)
 \right\rVert_{L^2(\P;\R^d)}\nonumber\\
&\leq \sqrt{2}(\sqrt{T}+1)cd^ce^{(\sqrt{T}+1)^2c^2T}
e^{2nc(\sqrt{T}+1)\sqrt{T} }\frac{e^{m/2}}{m^{n/2}}
+4.5
\sqrt{T/K}cd^ce^{4(\sqrt{T}+1)(T+1)c^2}.\label{c04}
\end{align}
This shows \eqref{a19c}.

\refstepcounter{step}\label{s08}
\emph{Step~\thestep.}\ We analyze the computational cost. Recall that
$\mathsf{C}^d_{n,m,K}$ the computational cost to construct the whole discrete process $(X_{n,m,K}^\theta(\frac{kT}{K}))_{k\in [1,K]\cap\Z}$ given that we have prepared a discrete Brownian path $W^\theta(\frac{kT}{K}))_{k\in [1,K]\cap\Z} $.
Note that by definition for all $d,n,m,K\in \N$ the cost $\mathsf{C}^d_{n,m,K}$ contains the following costs:
\begin{itemize}
\item The cost to calculate $\mu_d(0,0)$ is $\mathsf{cost}_{\mu_d}$.
\item The cost to calculate $\sigma_d(0,0)$ is $\mathsf{cost}_{\sigma_d}$.
\item For each $\ell \in [1,n-1]\cap\Z$, $k\in [1,m^{n-\ell}]\cap\Z$:
\begin{itemize}
\item We need to prepare two processes
$X^{d,\theta}_{\ell,m,K}$ and $X^{d,\theta}_{\ell-1,m,K}$, which brings the cost 
$\mathsf{C}^d_{\ell,m,K}+\mathsf{C}^d_{\ell-1,m,K}$.
\item We need to prepare two processes
$X^{d,(\theta,n,k,\ell)}_{\ell,m,K}$ and $X^{d,(\theta,n,k,\ell)}_{\ell-1,m,K}$, 
which come from a new source of Brownian randomness,
which brings the cost 
$\mathsf{C}^d_{\ell,m,K}+\mathsf{C}^d_{\ell-1,m,K}+ Kd\mathsf{cost}_{\mathsf{rv}}$. Here,
$Kd$ is the size of the Brownian motion and 
$\mathsf{cost}_{\mathsf{rv}}$ is the cost to generate a scalar random variable.
\item We need to prepare a scalar random  variable, $\mathfrak{u}^{(\theta,n,k,\ell)}$, which brings the cost $\mathsf{cost}_\mathsf{rv} $
\item We need to evaluate $\mu$ at two points which brings the cost $2\mathsf{cost}_{\mu_d}$
\item We need to evaluate $\sigma_d$ at 
each point of the two processes
$((X_{\ell,m,K}^{d,\theta} (\frac{kT}{T} ),X_{\ell,m,K}^{d,(\theta,n,k,\ell)} (\frac{kT}{T} )))_{k\in [1,K]}$ and 
$((X_{\ell-1,m,K}^{d,\theta} (\frac{kT}{T} ),X_{\ell-1,m,K}^{d,(\theta,n,k,\ell)} (\frac{kT}{T} )))_{k\in [1,K]}$ which brings the cost $2K\mathsf{cost}_{\sigma_d}$
\end{itemize}
\end{itemize}
To summary we have for all $d,m,n,K\in \N$ that
\begin{align} \begin{split} 
\mathsf{C}^d_{n,m,K}\leq \mathsf{cost}_{\mu_d}+\mathsf{cost}_{\sigma_d}+\sum_{\ell=1}^{n-1}m^{n-\ell}\Bigl(&
(\mathsf{C}^d_{\ell,m,K} + \mathsf{C}^d_{\ell-1,m,K})
+
(\mathsf{C}^d_{\ell,m,K} + \mathsf{C}^d_{\ell-1,m,K}+Kd\mathsf{cost}_{\mathsf{rv}})\\
&+\mathsf{cost}_\mathsf{rv}+2\mathsf{cost}_{\mu_d}+2K\mathsf{cost}_{\sigma_d}\Bigr).\end{split}
\end{align}
This and \eqref{c10} show for all $n,m,d,K\in \N$ that
\begin{align}
\mathsf{C}^d_{n,m,K}\leq cd^c+ \sum_{\ell=1}^{n-1}
m^{n-\ell} \left(2\mathsf{C}^d_{\ell,m,K}+2\mathsf{C}^d_{\ell-1,m,K}+cd^c2Kd\right)
\end{align}
and hence
\begin{align}
m^{-n}
\mathsf{C}^d_{n,m,K}
&\leq cd^cm^{-n}
+\sum_{\ell=1}^{n-1}\left(
2m^{-\ell}
\mathsf{C}^d_{\ell,m,K}
+2m^{-(\ell-1)}
\mathsf{C}^d_{\ell-1,m,K}+2cd^cKd\right)\nonumber\\
&=
 cd^cm^{-n}
+2cd^cKdn
+\sum_{\ell=1}^{n-1}\left(
2m^{-\ell}
\mathsf{C}^d_{\ell,m,K}
+2m^{-(\ell-1)}
\mathsf{C}^d_{\ell-1,m,K}\right).
\end{align}
This, the fact that 
$\forall\,d,m,K\colon\mathsf{C}^d_{0,m,K}=0$, and
\cite[Corollary~2.3]{HKN2022}
(applied 
for every $d,m,K\in \N$ 
with
$
(a_n)_{n\in \N_0}
\gets (m^{-n}
\mathsf{C}^d_{n,m,K})_{n\in \N_0}$,
$
\kappa\gets 2
$, $\lambda\gets 2$
$c_1\gets cd^c$,
$c_2\gets 2cd^cKd$,
$c_3\gets 0$,
$c_4\gets 0$,
$\beta\gets \frac{1+2+\sqrt{(1+2)^2+4\cdot 2}}{2}$ in the notation of 
\cite[Corollary~2.3]{HKN2022})
show for all $d,m,n,K\in \N$ that
\begin{align}
m^{-n}
\mathsf{C}^d_{n,m,K}\leq \frac{3}{2}4^ncd^c+\frac{3\cdot2cd^cKd(4^n-1)}{2(4-1)}
\leq \frac{3}{2}4^ncd^c+4^ncd^cKd\leq 8^ncd^cKd
.
\end{align}
and hence 
$
\mathsf{C}^d_{n,m,K}\leq (8m)^ncd^cKd
$.
Therefore, for all $n,d\in \N$ we have that
\begin{align}\label{c15}
\mathsf{C}^d_{n,n,n^n}\leq (8n)^ncd^cn^nd\leq (8n)^{2n}cd^{c+1}
\end{align}
and hence
\begin{align}
\mathsf{C}^d_{n+1,n+1,(n+1)^{(n+1)}}\leq [8(n+1)]^{2(n+1)}cd^{c+1}\leq (16n)^{2n+2}cd^{c+1}
=16^{2n+2}n^2n^{2n}cd^{c+1}.\label{c11}
\end{align}

\refstepcounter{step}\label{s07}
\emph{Step~\thestep.}\ We analyze the computational complexity.
For every $d\in \N$, $\epsilon\in (0,1)$ let 
\begin{align}
\mathsf{n}_{d,\epsilon} =\inf\!\left(\left \{
n\in \N\colon \sup_{k\in [n,\infty)\cap\Z}
\sup _{t\in [0,T]}
\left\lVert
X^{d,0}_{k,k,k^k}(t)-X^{d,0}_\infty(t)
\right\rVert_{L^2(\P;\R^d)}<\epsilon
\right\}\cup\{\infty\}\right).\label{c16}
\end{align}
Next, \eqref{a19c} shows that there exists $C\in [1,\infty)$ such that for all $d,n\in \N$ we have that
\begin{align}
\left\lVert 
X^{d,0}_{n,n,n^n}(t)-X_\infty^{d,0}(t)
\right\rVert_{L^2(\P;\R^d)}\leq \frac{2cd^cC^ne^{n/2}}{n^{n/2}}.\label{c13}
\end{align}
This proves for all $d\in \N$ that
\begin{align}
\lim_{n\to\infty}
\left\lVert 
X^{d,0}_{n,n,n^n}(t)-X_\infty^{d,0}(t)
\right\rVert_{L^2(\P;\R^d)}=0.
\end{align}
Therefore, for all $d\in \N$, $\epsilon\in (0,1)$ we have that $\mathsf{n}_{d,\epsilon}\in \N$. 
Next, for every $\delta\in (0,1)$ let
\begin{align}
\mathfrak{C}_\delta =c^6 \sup_{n\in\N}
\left[16^{2n+2}n^2(2C^ne^{n/2})^{5}n^{-\delta n/2}\right].\label{c14}
\end{align}
Observe for all $\delta\in (0,1)$ that
there exist $\gamma\in (0,\infty)$ such that
 $\mathfrak{C}_\delta<32\cdot 16^2n^2\gamma^nn^{-\delta n/2}$. Since $(n^{\delta n/2})_{n\in \N} $ increases faster than any exponential function, we obtain that $\mathfrak{C}_\delta<\infty$ for all $\delta\in (0,1)$.
Then \eqref{c11} and \eqref{c13} show for all
$d,n\in \N$, $\delta\in (0,1)$ that
\begin{align}
&
\mathsf{C}^d_{n+1,n+1,(n+1)^{(n+1)}}
\left(
\sup_{t\in [0,T]}\left\lVert 
X^{d,0}_{n,n,n^n}(t)-X_\infty^{d,0}(t)
\right\rVert_{L^2(\P;\R^d)}\right)^{4+\delta}
\nonumber\\
&\leq 16^{2n+2}n^2n^{2n}cd^{c+1}\left(\frac{2cd^cC^ne^{n/2}}{n^{n/2}}
\right)^{4+\delta}\nonumber\\
&\leq (cd^{c+1})^6\left[16^{2n+2}n^2(2C^ne^{n/2})^{5}n^{-\delta n/2}\right]\nonumber\\
&\leq (d^{c+1})^6\mathfrak{C}_\delta
<\infty.\label{c17}
\end{align}
This,
\eqref{c16}, \eqref{c15}, \eqref{c14}, and the fact that $C\geq 1$ show for all $d\in \N$, $\epsilon,\delta\in (0,1)$
that
\begin{align}
&
\mathsf{C}_{\mathsf{n}_{d,\epsilon},\mathsf{n}_{d,\epsilon}, (\mathsf{n}_{d,\epsilon})^{\mathsf{n}_{d,\epsilon}}}^d\epsilon^{4+\delta}\nonumber\\
&\leq\1_{\{1\}}(\mathsf{n}_{d,\epsilon})\mathsf{C}^d_{1,1,1}
+\1_{[2,\infty)} (\mathsf{n}_{d,\epsilon})
\mathsf{C}_{\mathsf{n}_{d,\epsilon},\mathsf{n}_{d,\epsilon}, (\mathsf{n}_{d,\epsilon})^{\mathsf{n}_{d,\epsilon}}}^d
\left(
\sup_{t\in [0,T]}\left\lVert 
X^{d,0}_{n,n,n^n}(t)-X_\infty^{d,0}(t)
\right\rVert_{L^2(\P;\R^d)}\right)^{4+\delta}|_{n=\mathsf{n}_{d,\epsilon}-1}\nonumber\\
&\leq\1_{\{1\}}(\mathsf{n}_{d,\epsilon}) 64 cd^{c+1}+
\1_{[2,\infty)} (\mathsf{n}_{d,\epsilon})
(d^{c+1})^6\mathfrak{C}_\delta\nonumber\\
&\leq (d^{c+1})^6\mathfrak{C}_\delta.\label{c18}
\end{align}
This, the fact that $\forall\,\delta\in (0,1)\colon \mathfrak{C}_\delta<\infty$, and \eqref{c16} show \eqref{c19}.
The proof of \cref{c01} is thus completed.
\end{proof}

\bibliographystyle{acm}
\bibliography{References}

\end{document}

%% file: ou_table.txt
10 & $L^2$-Error  & 2.8863 & 0.2909 & 0.0709 & 0.0149 & 0.0005 \\ 
 & Time  & $0.0001$  & $0.0011$  & $0.0562$  & $10.6896$  & $4082.6706$  \\ 
 & Cost  & $2.48 \cdot 10^{3}$  & $5.21 \cdot 10^{4}$  & $4.19 \cdot 10^{6}$  & $8.43 \cdot 10^{8}$  & $3.05 \cdot 10^{11}$  \\ 
\hline 
50 & $L^2$-Error  & 2.9584 & 0.3680 & 0.0799 & 0.0185 & 0.0003 \\ 
 & Time  & $0.0002$  & $0.0014$  & $0.0693$  & $13.5308$  & $5088.0725$  \\ 
 & Cost  & $2.60 \cdot 10^{5}$  & $5.47 \cdot 10^{6}$  & $4.39 \cdot 10^{8}$  & $8.84 \cdot 10^{10}$  & $3.20 \cdot 10^{13}$  \\ 
\hline 
100 & $L^2$-Error  & 2.8464 & 0.3654 & 0.0817 & 0.0176 & 0.0003 \\ 
 & Time  & $0.0075$  & $0.0057$  & $0.2641$  & $50.0334$  & $11458.5722$  \\ 
 & Cost  & $2.04 \cdot 10^{6}$  & $4.29 \cdot 10^{7}$  & $3.44 \cdot 10^{9}$  & $6.93 \cdot 10^{11}$  & $2.50 \cdot 10^{14}$  \\ 
\hline 
500 & $L^2$-Error  & 2.8415 & 0.3557 & 0.0814 & 0.0180 & 0.0004 \\ 
 & Time  & $0.0196$  & $0.1250$  & $6.4389$  & $1477.7328$  & $229473.3886$  \\ 
 & Cost  & $2.51 \cdot 10^{8}$  & $5.27 \cdot 10^{9}$  & $4.23 \cdot 10^{11}$  & $8.52 \cdot 10^{13}$  & $3.08 \cdot 10^{15}$  \\ 
\hline 
1000 & $L^2$-Error  & 2.8466 & 0.3482 & 0.0829 & 0.0178 & 0.0003 \\ 
 & Time  & $0.0582$  & $0.5611$  & $41.3946$  & $7831.0792$  & $1216612.0450$  \\ 
 & Cost  & $2.00 \cdot 10^{9}$  & $4.21 \cdot 10^{10}$  & $3.38 \cdot 10^{12}$  & $6.80 \cdot 10^{14}$  & $2.46 \cdot 10^{16}$  \\ 
\hline 

%% file: ku_table.txt
10 & $L^2$-Error  & 0.9851 & 0.2834 & 0.1497 & 0.0725 & 0.0564 \\ 
 & Time  & $0.0001$  & $0.0010$  & $0.0462$  & $8.5473$  & $3369.9302$  \\ 
 & Cost  & $2.01 \cdot 10^{3}$  & $4.22 \cdot 10^{4}$  & $3.40 \cdot 10^{6}$  & $6.83 \cdot 10^{8}$  & $2.47 \cdot 10^{11}$  \\ 
\hline 
50 & $L^2$-Error  & 0.9968 & 0.2891 & 0.1623 & 0.0673 & 0.0585 \\ 
 & Time  & $0.0002$  & $0.0012$  & $0.0603$  & $11.5126$  & $3914.4276$  \\ 
 & Cost  & $2.48 \cdot 10^{5}$  & $5.21 \cdot 10^{6}$  & $4.18 \cdot 10^{8}$  & $8.42 \cdot 10^{10}$  & $3.04 \cdot 10^{13}$  \\ 
\hline 
100 & $L^2$-Error  & 0.9778 & 0.2904 & 0.1519 & 0.0696 & 0.0685 \\ 
 & Time  & $0.0025$  & $0.0021$  & $0.0964$  & $18.2436$  & $6410.4096$  \\ 
 & Cost  & $1.99 \cdot 10^{6}$  & $4.18 \cdot 10^{7}$  & $3.36 \cdot 10^{9}$  & $6.76 \cdot 10^{11}$  & $2.44 \cdot 10^{14}$  \\ 
\hline 
500 & $L^2$-Error  & 1.0143 & 0.2940 & 0.1566 & 0.0699 & 0.0687 \\ 
 & Time  & $0.0036$  & $0.0344$  & $2.3884$  & $483.1741$  & $154138.4544$  \\ 
 & Cost  & $2.50 \cdot 10^{8}$  & $5.24 \cdot 10^{9}$  & $4.21 \cdot 10^{11}$  & $8.48 \cdot 10^{13}$  & $1.53 \cdot 10^{16}$  \\ 
\hline 
1000 & $L^2$-Error  & 1.0102 & 0.2944 & 0.1562 & 0.0703 & 0.0694 \\ 
 & Time  & $0.0245$  & $0.3807$  & $28.8793$  & $5413.7144$  & $1698069.5603$  \\ 
 & Cost  & $2.00 \cdot 10^{9}$  & $4.20 \cdot 10^{10}$  & $3.37 \cdot 10^{12}$  & $6.78 \cdot 10^{14}$  & $1.22 \cdot 10^{17}$  \\ 
\hline 